\documentclass[12pt,leqno]{amsart}

\usepackage[a4paper,left=2cm,top=2cm,right=2cm,bottom=2cm]{geometry}
\usepackage{amsmath,amssymb,amsthm,mathrsfs}
\usepackage{hyperref}
\usepackage{amsfonts}
\usepackage{tikz}
\usepackage{tikz-cd}
\usepackage{tkz-euclide}
\usepackage[normalem]{ulem}
\usepackage{bbm}
\usepackage{mathabx}
\usepackage{enumitem} 
\setenumerate[1]{label=(\roman{*}), ref=(\roman{*})}
\setenumerate[2]{label=(\alph{*}), ref=(\alph{*})}
\SetLabelAlign{Center}{\hfil#1\hfil}

\newcommand{\N}{\mathbb N}

\newcommand{\R}{\mathbb R}

\newcommand{\abs}[1]{\left\vert#1\right\vert}

\newcommand{\eps}{\varepsilon}

\newcommand{\1}{\mathbbm{1}}

\newcommand{\RNP}{Radon--Nikod\'{y}m property}

\newcommand{\norm}[1]{\left\Vert#1\right\Vert}
\newcommand{\nnorm}[1]{\left\vvvert#1\right\vvvert}

\newcommand{\conv}{\mathop{\mathrm{conv}}\nolimits}
\newcommand{\cconv}{\mathop{\overline{\mathrm{conv}}}\nolimits}
\newcommand{\diam}{\mathop{\mathrm{diam}}\nolimits}
\newcommand{\supp}{\mathop\mathrm{supp}}
\newcommand{\dent}{\mathop\mathrm{dent}}
\newcommand{\ext}{\mathop\mathrm{ext}}

\newcommand{\spn}{\mathop{\mathrm{span}}\nolimits}
\newcommand{\cspan}{\mathop{\overline{\mathrm{span}}}\nolimits}

\newcommand{\sign}{\mathop{\mathrm{sign}}\nolimits}

\newtheorem{theorem}{Theorem}[section]
\newtheorem*{theorem*}{Theorem}
\newtheorem{lemma}[theorem]{Lemma}
\newtheorem{question}[theorem]{Question}

\newtheorem{proposition}[theorem]{Proposition}
\newtheorem{corollary}[theorem]{Corollary}

\theoremstyle{definition}
\newtheorem{definition}[theorem]{Definition}
\newtheorem{example}[theorem]{Example}
\theoremstyle{remark}
\newtheorem{remark}[theorem]{Remark}
\numberwithin{equation}{section}

\begin{document}

\title{Unconditional bases and Daugavet renormings}

\author{Rainis Haller}
\address{Institute of Mathematics and Statistics, University of Tartu, Narva mnt 18, 51009 Tartu, Estonia}
\email{rainis.haller@ut.ee}
\urladdr{}

\author{\href{https://orcid.org/0000-0001-9649-7282}{Johann Langemets}}
\address{Institute of Mathematics and Statistics, University of Tartu, Narva mnt 18, 51009 Tartu, Estonia}
\email{johann.langemets@ut.ee}
\urladdr{\url{https://johannlangemets.wordpress.com/}}

\author{Yo\"el Perreau}
\address{Institute of Mathematics and Statistics, University of Tartu, Narva mnt 18, 51009 Tartu, Estonia}
\email{yoel.perreau@ut.ee}
\urladdr{}

\author{Triinu Veeorg}
\address{Institute of Mathematics and Statistics, University of Tartu, Narva mnt 18, 51009 Tartu, Estonia}
\email{triinu.veeorg@ut.ee}
\urladdr{}

\thanks{This work was supported by the Estonian Research Council grants PRG1901, PSG487 and SJD58.}
\subjclass[2020]{}

\begin{abstract} 

We introduce a new diametral notion for points of the unit sphere of Banach spaces, that naturally complements the notion of $\Delta$-points, but is weaker than the notion of Daugavet points. We prove that this notion can be used to provide a new geometric characterization of the Daugavet property, as well as to recover -- and even to provide new -- results about Daugavet points in various contexts such as absolute sums of Banach spaces or projective tensor products. Finally, we show that this notion leads to powerful new ideas for renorming questions, and that those ideas can be combined with previous constructions from the literature in order to renorm every infinite dimensional Banach space with an unconditional Schauder basis to have a Daugavet point.
    
\end{abstract}

\maketitle



\section{Introduction}\label{section:introduction}

In this note, we study pointwise versions of the Daugavet property (DPr) and of various diametral diameter 2 properties (DD2P, DLD2P and property $\mathfrak{D}$) that were introduced in \cite{AHLP} and \cite{MPRZ}. If $X$ is a Banach space, then we denote respectively by $B_X$ and $S_X$ the closed unit ball  and the unit sphere of $X$. We also denote by $X^*$ the topological dual of $X$. For simplicity, we deal with real Banach spaces only.

For a given $x\in S_X$, we write $D(x)$ the set of all norm-one \emph{supporting functionals} at $x$, that is the set of all $x^*\in S_{X^*}$ which satisfy $x^*(x)=1$. By a \emph{slice} of $B_X$, we mean any non-empty subset of $B_X$ obtained by intersecting $B_X$ with an open half-space. Any such set can be written in the form \begin{equation*}
    S(x^*,\alpha):=\{x\in B_X:\ x^*(x)>1-\alpha\}
\end{equation*} for some $x^*\in S_{X^*}$ and $\alpha>0$. We will denote this set by $S(B_X,x^*,\alpha)$  whenever we want to specify which unit ball we are slicing. By a \emph{relatively weakly open subset} of $B_X$, we mean any subset of $B_X$ obtained by intersecting $B_X$ with an open subset of $X$ for the weak topology. 

\begin{definition}[{\cite{AHLP, MPRZ}}]\label{defn:diametral_points}

    Let $X$ be a Banach space, and let $x\in S_X$. We say that $x$ is 
    \begin{enumerate}
        \item\label{item:super-Daugavet} a \emph{super Daugavet point} if $\sup_{y\in V}\norm{x-y}=2$ for every non-empty relatively weakly open subset $V$ of $B_X$;

        \item\label{item:super-Delta} a \emph{super $\Delta$-point} if $\sup_{y\in V}\norm{x-y}=2$ for every non-empty relatively weakly open subset $V$ of $B_X$  that contains $x$;

        \item\label{item:Daugavet} a \emph{Daugavet point} if $\sup_{y\in S}\norm{x-y}=2$ for every slice $S$ of $B_X$;

        \item\label{item:Delta} a \emph{$\Delta$-point} if $\sup_{y\in S}\norm{x-y}=2$ for every slice $S$ of $B_X$  that contains $x$;

        \item\label{item:Dirk} a \emph{$\mathfrak{D}$-point} if $\sup_{y\in S(x^*,\alpha)}\norm{x-y}=2$ for every $x^*\in D(x)$ and $\alpha>0$.
    
    \end{enumerate}
\end{definition}

We will also consider natural weak$^*$ version of the points \ref{item:super-Daugavet}--\ref{item:Delta} in dual spaces, where slices and relatively weakly open subsets are respectively replaced by weak$^*$ slices (i.e. slices defined by an element of the predual) and relatively open subsets for the weak$^*$ topology.

Finite dimensional spaces contain no $\mathfrak{D}$-points by the results from \cite{ALMP} (see \cite[Corollary~5.3]{AALMPPV}), but it was recently shown in \cite[Theorem~4.1]{AALMPPV} that in infinite dimension, $\Delta$-points are very much an isometric notion,  as every infinite dimensional Banach space can be renormed with a $\Delta$-point, and as every infinite dimensional Banach space that fails the Schur property can be renormed with a super $\Delta$-point.

It was also proved in \cite[Theorem~2.1]{AALMPPV} that the Lipschitz-free space with the \RNP\ and with a Daugavet point that was constructed in \cite{VeeorgStudia} is isomorphic to $\ell_1$ and is isometrically a dual space. However, the question of whether infinite dimensional reflexive or super-reflexive Banach spaces could contain Daugavet points was left open. In this note, we fill this gap by constructing a renorming of $\ell_2$ with a super Daugavet point. More precisely, the following theorem is the main result of the present note.

\begin{theorem}\label{intro-thm:super_renorming}
    Let $X$ be an infinite dimensional Banach space with an unconditional weakly null Schauder basis $(e_n)$ and biorthogonal functionals $(e_n^*)$. Then there exists an equivalent
    norm $\nnorm{\cdot}$ on $X$ such that \begin{enumerate}
    
        \item $e_1$ is a super Daugavet point in $(X,\nnorm{\cdot})$;

        \item $e_1^*$ is a weak$^*$ super Daugavet point in $(E,\nnorm{\cdot})$, where $E:=\cspan\{e_n^*\}$.

    \end{enumerate}
     
In particular, if $(e_n)$ denotes the unit vector basis of $\ell_2$, then there exists a renorming of $\ell_2$ for which $e_1$ is a super Daugavet point in the new norm and its dual norm. 

\end{theorem}

Combining this result, the $\ell_1$-renorming from \cite{AALMPPV}, classic results from James about unconditional bases, and some observations on complemented subspaces, we will ultimately get the following statement. 

\begin{theorem}\label{intro-thm:unconditional_renorming}

Every infinite dimensional Banach space that contains a complemented unconditional basic sequence can be renormed with a Daugavet point and a weak$^*$ Daugavet point in its dual.
    
\end{theorem}

In particular, we provide a positive answer to \cite[Question~6.1]{ALMP} where it was asked whether super-reflexive Banach spaces could contain Daugavet points; and a partial answer towards the negative to \cite[Question~7.7]{MPRZ} where it was asked whether an isomorphic obstruction for Daugavet or super Daugavet points could be found. We will also give a positive answer to \cite[Question~5.29]{AALMPPV} by constructing a non-reflexive M-embedded space with a super Daugavet point.

Some of the key ideas behind the construction from Theorem~\ref{intro-thm:super_renorming} actually came from the study of another diametral notion for points of the unit sphere, that has not yet been considered in the literature, but that is a natural complement to the notion of $\Delta$-points.

\begin{definition}\label{defn:nabla-points}

Let $X$ be a Banach space, and let $x\in S_X$. We say that $x$ is a \emph{$\nabla$-point} if $\sup_{y\in S}\norm{x-y}=2$ for every slice $S$ of $B_X$ that does not contain $x$.
    
\end{definition}

Clearly, a point $x$ in the unit sphere of a Banach space $X$ is a Daugavet point if and only if it is simultaneously a $\nabla$-point and a $\Delta$-point. We shall see in the following section that, in fact, $x$ is already a Daugavet point if it is a $\nabla$-point and a $\mathfrak{D}$-point.

With this in mind, our primary motivation for the introduction of the notion of $\nabla$-points was to look for new insights on the influence of Daugavet points for the geometry of Banach spaces. Indeed, the majority of the obstructions that have been recently produced in the literature for the existence of diametral points (e.g. in \cite{ALMP,VeeorgFunc,KLT,MPRZ,AALMPPV}) actually came from an obstruction to the various $\Delta$-conditions. Among the few exceptions where the considerations about absolute sums from \cite{AHLP}, and the distance 2 to denting points that was observed in \cite[Proposition~3.1]{JRZ2022}. We will state analogues of those results for $\nabla$-points, hence show that there are indeed cases where this notion is the one preventing the existence of Daugavet points. 

It is of course natural to ask whether there are $\nabla$-points that are not Daugavet points. The following easy but fundamental examples show that there are indeed plenty of those, and illustrate why the notion is, in a way, a bit too flexible to negate good geometric properties of Banach spaces. We will see other examples of that kind throughout the text. However, let us point out that the natural super version of the $\nabla$-condition would not provide a new notion, as one can easily show that any ``super $\nabla$-point" is automatically a super Daugavet point (this immediately follows from the fact that the weak topology is Hausdorff).

\begin{example}\label{expl:l1_nabla}

  The elements of the unit vector basis of $\ell_1^n$ and $\ell_1$ are all $\nabla$-points.
    
\end{example}

In particular, note that every slice of $B_{\ell_1^n}$ or $B_{\ell_1}$ contains a $\nabla$-point, while neither $\ell_1^n$ nor $\ell_1$ contains $\Delta$-points by \cite[Theorem~3.1]{AHLP} (although every element with infinite support in $S_{\ell_1}$ is a $\mathfrak{D}$-point by \cite[Proposition~2.3]{AHLP}). In view of those examples, one can wonder whether a Banach space whose unit sphere is entirely composed of $\nabla$-points has to contain Daugavet points. Our main result concerning $\nabla$-points is that this property is actually equivalent to the DPr.

\begin{theorem}\label{intro-thm:DPr=nabla-DPr}

    Let $X$ be a Banach space of dimension $\dim{X}>1$. Then $X$ has the DPr if and only if every point on its unit sphere is a $\nabla$-point.
    
\end{theorem}

Let us end the present section by giving a few words about the organization and the content of the paper. In Section~\ref{section:nabla_DPr}, we make a general study of the notion of $\nabla$-points. We start by proving that any point that is simultaneously a $\nabla$-point and  a $\mathfrak{D}$-point is a Daugavet-point, and deduce that a $\nabla$-point is always either a Daugavet point or a strongly exposed point. Building on these results, we prove Theorem~\ref{intro-thm:DPr=nabla-DPr} and obtain a new characterization of the DPr in terms of $\nabla$-points. Finally, we state a specific property of distance to denting points for $\nabla$-points, and observe that any strictly convex space which contains a $\nabla$-point fails the \RNP.

In Section~\ref{section:examples}, we study $\nabla$-points in some classical Banach spaces. In Subsection~\ref{subsec:absolute_sums}, we look at $\nabla$-points in absolute sums of Banach spaces. Fist, we prove that if $N$ is an absolute normalized norm on $\R^2$ different from the $\ell_1$-norm and the $\ell_\infty$-norm, then the notions of $\nabla$-points and of Daugavet points coincide in the absolute sum $X\oplus_N Y$ for arbitrary Banach spaces $X$ and $Y$. Second, we make a specific study of $\nabla$-points in $\ell_1$-sums and $\ell_\infty$-sums of Banach spaces, and prove that in order for a point $(x,y)$ in $X\oplus_N Y$ to be a $\nabla$-point without being a Daugavet point, we must have either $N=\norm{\cdot}_1$ and $(\norm{x},\norm{y})=(1,0)$ or $(0,1)$, or $N=\norm{\cdot}_\infty$ and $(\norm{x},\norm{y})=(1,1)$. Last, we provide a complete characterization of those points in this context.
In Subsection~\ref{subsection:function-spaces}, we consider $\nabla$-points in some specific function spaces. On the one hand, we show that $L_1(\mu)$-spaces provide a natural framework for the existence of $\nabla$-points that are not Daugavet points. On the other, we show that the notion of $\nabla$-points coincides with all the other diametral notions in every infinite dimensional $C(K)$-space or $C_0(L)$-space, as well as in general $L_1$-preduals. In Subsection~\ref{subsec:tensor_products}, we present a few non-trivial applications of the notion of $\nabla$-points in tensors products, where known transfer results for $\Delta$-points can be combined with properties of $\nabla$-points to get new transfer results for Daugavet points. Finally, we provide partial specific transfer results for $\nabla$-points, and start an investigation of the stability of super points.

In Section~\ref{section:renormings}, we deal with Daugavet renormings. First we prove that it naturally follows from some of the transfer results through $\ell_1$-sums that every Banach space can be renormed with a $\nabla$-point. Second, we combine this idea with an adapted version of the renorming from \cite[Section~3]{AALMPPV} to get the main result of the text, Theorem~\ref{intro-thm:super_renorming}. Last, we observe that if a Banach space $X$ contains a complemented subspace $Y$ that can be renormed with a Daugavet or a super Daugavet point, then $X$ can also be renormed with a Daugavet or a super Daugavet point. Based on this, we combine Theorem~\ref{intro-thm:super_renorming} with classic results from James about unconditional bases and the $\ell_1$-isomorphism from \cite{AALMPPV} to get that every infinite dimensional Banach space with an unconditional basis (and more generally with a complemented unconditional basic sequence) can be renormed with a Daugavet point. 

Throughout the text, we will use standard notation and will usually follow the textbooks \cite{AlbiacKalton} or \cite{Czechbook}. In particular, if $A$ is a non-empty subset of a normed space $X$, then we will denote respectively by $\spn{A}$ and $\conv{A}$ the linear span and the convex hull of $A$; and by $\cspan{A}$ and $\cconv{A}$ their respective closure. We will also denote by $\ext{B_X}$ the set of all extreme points of $B_X$, and by $\dent{B_X}$ the set of all denting points of $B_X$.


\section{\texorpdfstring{$\nabla$}{Nabla}-points and the Daugavet property}\label{section:nabla_DPr} 

We start with the following easy observations.

\begin{lemma}\label{lem:Daugavet=nabla+D}

    Let $X$ be a Banach space and let $x\in S_X$. Then $x$ is a Daugavet point if and only if it is simultaneously a $\nabla$-point and a $\mathfrak{D}$-point. 

\end{lemma}

\begin{proof}

    Clearly, a Daugavet point is simultaneously a $\nabla$-point and a $\Delta$-point, hence a $\nabla$-point and a $\mathfrak{D}$-point. Conversely, assume that $x$ is simultaneously a $\nabla$-point and a $\mathfrak{D}$-point, and let $S:=S(x^*,\alpha)$ with $x^*\in S_{X^*}$ and $\alpha>0$. If $x^*\in D(x)$, then $\sup_{y\in S}\norm{x-y}=2$ because $x$ is a $\mathfrak{D}$-point. On the other hand, if $x^*\notin D(x)$, then there exists $\beta>0$ such that $x^*(x)\leq1-\beta$. If $\alpha\leq \beta$, then $x\notin S$,  and $\sup_{y\in S}\norm{x-y}=2$ because $x$ is a $\nabla$-point. If $\alpha>\beta$, then $S(x^*,\beta)\subset S$, so $\sup_{y\in S}\norm{x-y}\geq \sup_{y\in S(x^*,\beta)}\norm{x-y}=2$ by the previous case. It follows that $x$ is a Daugavet point. 
    
\end{proof}

\begin{proposition}\label{prop:closed_nabla-set}

    Let $X$ be a Banach space. The set of all $\nabla$-points in $X$ is a closed subset of $X$. 
    
\end{proposition}

\begin{proof}
    The verification is immediate by the definition of $\nabla$-points. Indeed, assume that $(x_n)$ is a sequence of $\nabla$-points in $X$ which converges to some $x\in X$. Consider a slice $S$ of $B_X$ which does not contain $x$ and let $\varepsilon>0$. Since $x_n\to x$, we can find $n\in\N$ such that $x_n\notin S$ and $\|x-x_n\|\leq \eps/2$. Then as $x_n$ is a $\nabla$-point, there exists $y\in S$ such that $\|x_n-y\|\geq 2-\eps/2$. Thus $\|x-y\|\geq 2-\eps$, and it follows $x$ is also a $\nabla$-point.
    
\end{proof}

We will now show that a $\nabla$-point is always either a Daugavet point or a strongly exposed point of the unit ball. 

\begin{theorem}\label{thm:Nabla_not_str-exp_is_Daugavet}
    Let $X$ be a Banach space, and let $x\in S_X$ be a $\nabla$-point. Then either $x$ is a Daugavet point, or $x$ is a strongly exposed point of $B_X$.
\end{theorem}

\begin{proof}
    Assume that $x$ is not a strongly exposed point of $B_X$. To prove that $x$ is a Daugavet point, it suffices by Lemma~\ref{lem:Daugavet=nabla+D} to show that $x$ is a $\mathfrak{D}$-point. 
    So fix $\alpha,\eps>0$ and $x^*\in D(x)$, and let us assume as we may that $\alpha< 1/2$. We want to show that there exists  $y\in S(x^\ast,\alpha)$ such that $\|x-y\|\ge 2-\varepsilon$.
    By assumption, there exist $\beta\in(0,\alpha)$ such that $\diam\big(S(x^*,\gamma)\big)>4\beta$ for every $\gamma>0$. So let $\gamma>0$ be such that 
    \begin{equation*}
        1-\frac{\alpha-\gamma}{2}<\frac{\beta}{\gamma+\beta},
    \end{equation*}
    and pick $z\in B_X$ such that $x^\ast(z)>1-\gamma$ and $\|x-z\|\ge 2\beta$. Then let $y^\ast\in \frac12 S_{X^\ast}$ be such that $y^\ast(z)-y^\ast(x)\ge \beta$, and fix $\lambda\in (0,1)$ such that 
    \begin{equation*}
    \lambda (1+\alpha-\gamma)+(1-\lambda)y^\ast(z)=1.
    \end{equation*}
    Note that 
    \begin{equation*}
        \lambda=\frac{1-y^*(z)}{1+\alpha-\gamma-y^*(z)}=1-\frac{\alpha-\gamma}{1+\alpha-\gamma-y^*(z)}<1-\frac{\alpha-\gamma}{2}<\frac{\beta}{\gamma+\beta}.
    \end{equation*}
    Set $z^\ast=\lambda x^\ast+(1-\lambda)y^\ast$.
    Then 
    \begin{align*}
        z^*(x)&=\lambda x^\ast(x)+(1-\lambda)y^\ast(x)\le\lambda +(1-\lambda)(y^*(z)-\beta)\\
        &=\lambda+1-\lambda(1+\alpha-\gamma)-(1-\lambda)\beta=1-\lambda\alpha+\lambda\gamma+\lambda \beta-\beta<1-\lambda\alpha
    \end{align*}
    and 
    \begin{equation*}z^*(z)=\lambda x^*(z)+(1-\lambda)y^*(z)>\lambda(1-\gamma)+1-\lambda(1+\alpha-\gamma)=1-\lambda\alpha.\end{equation*}
    Since $x$ is a $\nabla$-point, there exists $y\in B_X$ such that $z^*(y)>1-\lambda\alpha$ and $\|x-y\|\ge 2-\varepsilon$. Now $y\in S(x^\ast,\alpha)$ because
    \begin{equation*}\lambda x^*(y)=z^*(y)-(1-\lambda)y^*(y)>(1-\lambda\alpha)-(1-\lambda)=\lambda(1-\alpha),\end{equation*} so we are done.
\end{proof}

Using the previous result, we can now prove that main theorem of the section.

\begin{theorem}\label{thm:DPr=nabla-DPr}

    A Banach space $X$ of dimension $\dim{X}>1$ has the DPr if and only if every element on its unit sphere is a $\nabla$-point.
    
\end{theorem}

\begin{proof}
    Since every Daugavet point is $\nabla$-point, one implication is clear. So assume that every element $x\in S_X$ is a $\nabla$-point. In order to get that $X$ has the DPr, it now suffices  to show that $X$ has the LD2P (i.e. that every slice of $B_X$ has diameter two). Indeed, $B_X$ would then have no strongly exposed point, and by Theorem~\ref{thm:Nabla_not_str-exp_is_Daugavet} all unit sphere elements would be Daugavet points.
    
    So let $S:=S(x^*,\alpha)$ be a slice of the unit ball. As $\dim{X}>1$, there exists $x\in S\cap S_X$ such that $x^*(x)<1$. Let $\gamma\in (0,\alpha)$ be such that $x^*(x)<1-\gamma$. Then $x\notin S(x^*,\gamma)$, and thus $\sup_{y\in S(x^*,\gamma)}\norm{x-y}=2$ because $x$ is a $\nabla$-point. In particular, since $S(x^*,\gamma)\subset S$ and $x\in S$,  we get that $S$ has diameter 2, and $X$ has the LD2P, hence the DPr.
\end{proof}

\begin{remark}

    Analogously to \cite[Lemma~2.2]{AHLP}, one can easily prove that a point $x$ in the unit sphere of a Banach space $X$ is a $\nabla$-point if and only if the operator $T:=x^*\otimes x$ satisfies the Daugavet equation $\norm{Id-T}=2$ for every $x^*\in S_{X^*}$ that does not belong to $D(x)$. So  Theorem~\ref{thm:DPr=nabla-DPr} can be rephrased as: A Banach space of dimension greater than or equal to 2 has the DPr if and only if every rank-one norm-one operator that is not a projection satisfies the Daugavet equation. 
\end{remark}

We end the section by stating the following straightforward $\nabla$-analogue to \cite[Proposition~3.1]{JRZ2022} and by collecting a few applications of this result. 

\begin{proposition}\label{prop:nabla_distance_2_to_denting}
    Let $X$ be a Banach space and let $x\in S_X$ be a $\nabla$-point. Then $\|x-y\|=2$ for every denting point $y$ of $B_X$ with $y\neq x$.
\end{proposition}

In \cite{Kadets96}, an example of a strictly convex normed space with the DPr was provided. However, this space is not complete, and it is still an open question whether there exists a strictly convex Banach space with the DPr.  It was also asked in \cite[Question~7.5]{MPRZ} whether one could provide a strictly convex Banach space with a Daugavet point. In the following result, we observe that it immediately follows from Proposition~\ref{prop:nabla_distance_2_to_denting} that in order for a strictly convex space to contain a $\nabla$-point, its unit ball must contain no denting point other than the point itself and its opposite. In particular, strictly convex spaces with the \RNP\ cannot contain $\nabla$-points. 

\begin{corollary}

    Let $X$ be a strictly convex space and let $x\in S_X$. If $x$ is a $\nabla$-point, then $\dent{B_X}\subset\{\pm x\}$. 
    
\end{corollary}

\begin{proof}
    Since $X$ is strictly convex, we have that \begin{equation*}
        \{y\in B_X:\ \norm{x-y}=2\}=\{-x\}
    \end{equation*}  for every $x\in S_X$. So if $x$ is a $\nabla$-point, then it follows from Proposition~\ref{prop:nabla_distance_2_to_denting} that the only possible denting points in $B_X$ are $\pm x$.
\end{proof}

Recall that a Banach space $X$ is \emph{weakly midpoint locally uniformly rotund} (wMLUR) if every point in $S_X$ is weakly strongly extreme, a.k.a. preserved extreme. In \cite{AHNTT}, MLUR hence wMLUR Banach spaces with the DD2P were constructed. We do not know whether any of the spaces $X_D$ from \cite[Theorem~2.3]{AHNTT} contains a Daugavet point, but let us point out  that in this context, it would be enough to show that one of these contains a $\nabla$-point.

\begin{proposition}\label{prop:wMLUR_nabla}

Let $X$ be an infinite dimensional wMLUR Banach space. Then every $\nabla$-point in $X$ is super Daugavet.

\end{proposition}

\begin{proof}

    By Choquet's lemma, slices form bases of neighborhoods  in the relative weak topology for the preserved extreme points of the unit ball of any given Banach space. It immediately follows that every $\nabla$-point in a wMLUR space $X$ is a ``super $\nabla$-point", hence, as was previously observed, a super Daugavet point.
    
\end{proof}

Finally, recall that in a Banach space with the \RNP, every slice contains a denting point. Thus we immediately get, in those spaces, the following characterization for $\nabla$-points.

\begin{proposition}\label{prop:nabla-characterization_RNP}
    Let $X$ be a Banach space with the \RNP, and let $x\in S_X$. Then $x$ is a $\nabla$-point in $X$ if and only if $\|x-y\|=2$ for every denting point $y$ of $B_X$ with $y\neq x$.
\end{proposition}

Analogously to Daugavet points (see \cite[Theorem~2.1]{VeeorgStudia}), this characterization does also hold in every Lipschitz-free space.

\begin{proposition}
    Let $M$ be a metric space and let $\mu\in S_{\mathcal{F}(M)}$. Then $\mu$ is a $\nabla$-point if and only if $\|\mu-\nu\|=2$ for every denting point $\nu$ of $B_{\mathcal{F}(M)}$ with $\nu\neq \mu$.
\end{proposition}

\begin{proof}
One implication is covered by Proposition \ref{prop:nabla_distance_2_to_denting}. So assume that $\|\mu-\nu\|=2$ for every denting point $\nu$ of $B_{\mathcal{F}(M)}$ with $\nu\neq \mu$. Fix $\varepsilon>0$ and a slice $S(f,\alpha)$ with $\mu\notin S(f,\alpha)$. If $f$ is local, then by {\cite[Theorem~2.6]{JRZ2022}} there exists $m_{uv}\in S(f,\alpha)$ such that $\|\mu-m_{uv}\|\ge 2-\varepsilon$.
If $f$ is not local, then by {\cite[Proposition~2.7]{VeeorgFunc}} there exist a
denting point $m_{uv}\in S(f,\alpha)$, and by
assumption we have $\|\mu-m_{uv}\|=2$. Therefore $\mu$ is a 
$\nabla$-point.
\end{proof}


\section{\texorpdfstring{$\nabla$}{Nabla}-points in classical Banach spaces}\label{section:examples}

\subsection{Absolute sums}\label{subsec:absolute_sums}

Recall that a norm $N$ on $\R^2$ is said to be \emph{absolute} if $N(a,b)=N(\abs{a},\abs{b})$ for every $(a,b)\in \R^2$ and \emph{normalized} if $N(0,1)=N(1,0)=1$. If $X$ and $Y$ are Banach spaces, and if $N$ is an absolute normalized norm on $\R^2$, then we denote by $X\oplus_NY$ the \emph{absolute sum} of $X$ and $Y$, that is the Banach space  $(X\times Y,\norm{\cdot})$ where \begin{equation*}
    \norm{(x,y)}=N(\norm{x},\norm{y})\ \text{for every $(x,y)\in X\times Y$}.
\end{equation*} In particular, if $N:=\norm{\cdot}_p$ for some $p\in[1,\infty]$, then we simply denote by $X\oplus_p Y$ the \emph{$\ell_p$-sum} of $X$ and $Y$.

The study of Daugavet points in absolute sums of Banach spaces was started in \cite{AHLP} and completed in \cite{HPV}. In particular, recall that no $\ell_p$-sum of Banach spaces contains a Daugavet point when $1<p<\infty$, and that more generally, the absolute sum $X\oplus_N Y$ does not contain a Daugavet point if $N$ has the property $(\alpha)$ from \cite[Definition~4.4]{AHLP} for arbitrary $X$ and $Y$. Also recall that Daugavet points transfer very well through $\ell_1$-sums and $\ell_\infty$-sums, and that positive transfer results are more generally available for A-octaheral norms, see \cite[Section~2 and Section~3]{HPV}.

As was noticed in \cite[Section~4]{AHLP}, $\Delta$-points are much more flexible with respect to this operation than Daugavet points, and thus the $\Delta$-condition is not the one that provides obstructions to the existence of Daugavet points in absolute sums of Banach spaces. In fact, it can easily be checked that the $\nabla$-condition is the one that prevents the existence of these points whenever such an obstruction exists. But more can be said. Indeed, we will actually show that in absolute sums of Banach spaces, Daugavet points and $\nabla$-points can be different only if they are related either to the points $(1,0)$ or $(0,1)$ in $\ell_1^2$, or to the point $(1,1)$ in $\ell_\infty^2$. 

The specificity of these points is clearly apparent in the finite dimensional examples. 
Indeed, we have seen in Example~\ref{expl:l1_nabla} that every element of the unit vector basis $(e_i)$ of $\ell_1^n$ is a $\nabla$-point, so elements of the form $(x,0)$ or $(0,y)$ with $x\in S_X$ and $y\in S_Y$ can be $\nabla$-points without being Daugavet points in an $\ell_1$-sum. On the contrary, the point $\frac{1}{2}(e_1+e_2)$ is not a $\nabla$-point in $\ell_1^2$, so the $\ell_1$-sum of two $\nabla$-points with respect to the point $(\frac{1}{2},\frac{1}{2})$ need not be a $\nabla$-point. Similarly, the following example shows that elements of the form $(x,y)$ with $x\in S_X$ and $y\in S_Y$ can be $\nabla$-points without being Daugavet points in an $\ell_\infty$-sum, while the $\ell_\infty$-sum of $0$ and a $\nabla$-point need not be a $\nabla$-point, as $e_1$ is not $\nabla$ in $\ell_\infty^2$.

\begin{example}\label{expl:l_infty^n_nabla}

For every $n\in \N$ and for every $\theta:=(\theta_i)\in\{-1,1\}^n$, we have that $\sum_{i=1}^n\theta_i e_i$ is a $\nabla$-point in $\ell_\infty^n$.
    
\end{example}

So let us start by proving that the notions of Daugavet and $\nabla$-points coincide whenever the underlying absolute norm is neither equal to the $\ell_1$-norm nor to the $\ell_\infty$-norm.

\begin{theorem}\label{thm:Daugavet=Nabla_absolute-sums_not-1-infty}
Let $X$ and $Y$ be Banach spaces, let $N$ be an absolute normalized norm that is different from the $\ell_1$-norm and the $\ell_\infty$-norm, and let $(x,y)\in S_{X\oplus_NY}$. Then $(x,y)$ is a $\nabla$-point in $X\oplus_NY$ if and only if it is a Daugavet point. 
\end{theorem}

\begin{proof}
    As every Daugavet point is a $\nabla$-point, it suffices to show that $(x,y)\in S_{X\oplus_NY}$ is a Daugavet point whenever it is assumed to be a $\nabla$-point. So assume that $(x,y)\in S_{X\oplus_NY}$ is a $\nabla$-point. We will start by proving that if $x\neq0$, then $x/\|x\|$ is a Daugavet point. 
    
    First notice that if $0<\|x\|<1$, then we can actually copy paste the proof of \cite[Theorem~3.1]{HPV}. Indeed, the key slices that are involved there are defined by functionals of the form $f:=(x^*,0)$ with $x^*\in S_{X^*}$. But as we have $f(x,y)\le \|x\|<1$ by assumption, taking a small enough parameter $\delta>0$ will exclude $(x,y)$ from the corresponding slices, and the $\nabla$-condition will be available in place of the Daugavet condition. We leave the details to the reader.
    
   Now assume that $\|x\|=1$. Then $\|y\|<1$ because $N$ is different from the $\ell_\infty$-norm. Fix $x^*\in S_{X^*}$ and $\alpha>0$. Then let $y^*\in S_{Y^*}$ be such that $y^*(y)\le0$, and let $f:=(x^*,y^*)$. Since $N$ is different from the $\ell_1$-norm, we have $\|f\|>1$. Let $\varepsilon\in (0,1]$ be such that $\|f\|>1+\varepsilon$. By \cite[Lemma~1.4]{HPV} there exists $\delta\in \big(0,\varepsilon\alpha\big)$ such that $\|y\|<1-\delta$, $\|f\|>1+\varepsilon+\delta$, and for every $p,q,r\geq 0$, if
    \[2-\delta \le N(p,q)\leq N(r,q) \leq 2\quad \textnormal{and}\quad  q<2-\delta,\]
    then $|p-r|<\varepsilon^2$.
    As $y^*(y)\leq 0$, we have $f(x,y)\le x^*(x)\leq 1$, and thus since $(x,y)$ is a $\nabla$-point, there exists $(u,v)\in B_{X\oplus_NY}$ such that $f(u,v)>\|f\|-\delta$ and $\big\|(x,y)-(u,v)\big\|\ge 2-\delta$. Then
    \begin{equation*}
        \|u\|\ge x^*(u)=f(u,v)-y^*(v)>\|f\|-\delta-\|v\|\ge \|f\|-\delta-1\ge \varepsilon.
    \end{equation*}
     We also have 
    \begin{equation*}
        x^*(u)>\|f\|-\delta-\|v\|\ge \|x^*\|\|u\|+\|y^*\|\|v\|-\delta-\|v\|= \|u\|-\delta\ge \|u\|-\varepsilon\alpha\ge \|u\|(1-\alpha),
    \end{equation*}
    and thus $u/\|u\|\in S(x^*,\alpha)$.
    Furthermore
    \begin{equation*}
        2-\delta\le N\big(\|x-u\|,\|y-v\|\big)\le N\big(\|x\|+\|u\|,\|y-v\|\big)\le 2
    \end{equation*}
    and $\|y-v\|\le\|y\|+\|v\|< 2-\delta$.
    Hence  
    \begin{equation*}
        \big|\|x\|+\|u\|-\|x-u\|\big|<\varepsilon^2<\varepsilon\|u\|,
    \end{equation*}
    and
    \begin{align*}
        \Big\|x-\frac{u}{\|u\|}\Big\|&\ge \frac{1}{\|u\|}\|x-u\|-\Big(\frac{1}{\|u\|}-1\Big)\|x\|\\
        &\ge \frac{1}{\|u\|}\big(\|x\|+\|u\|-\varepsilon\|u\|\big)-\Big(\frac{1}{\|u\|}-1\Big)\|x\|\\
        &=2-\varepsilon.
    \end{align*}
    Thus $x$ is a Daugavet point. Similarly we can prove that if $y\neq0$, then $y/\|y\|$ is a Daugavet point. Consequently, we get by the observations from \cite[Section~4]{AHLP} that $(x,y)$ is a $\Delta$-point in $X\oplus_N Y$, and since it is also $\nabla$ by assumption, we deduce that $(x,y)$ is a Daugavet point as we wanted. 
\end{proof}

From this we immediately get as corollaries to the results from \cite{AHLP} and \cite{HPV} that $\ell_p$-sums of Banach spaces with $p\in(1,\infty)$ (as well as $N$-sums of Banach spaces when $N$ has property $(\alpha)$) do not admit $\nabla$-points, and that characterizations for $\nabla$-points are available in the case of A-octahedral norms that are different from the $\ell_1$-norm and the $\ell_\infty$-norm, see e.g. \cite[Theorems~2.2 and 3.1]{HPV}.

To conclude the section, it now only remains to do a specific study for $\ell_1$-sums and $\ell_\infty$-sums.  For $\ell_1$-sums, we have the following statements. 

\begin{proposition}\label{prop:nabla_points_l1_sums}
    Let $X$ and $Y$ be Banach spaces, and let $x\in S_X$. Then  $x$ is a $\nabla$-point in $X$ if and only if $(x,0)$ is a $\nabla$-point in $X\oplus_1 Y$.
\end{proposition}
\begin{proof}
    Let $Z:=X\oplus_1 Y$, and let us first assume that $x$ is a $\nabla$-point in $X$. Fix $\eps>0$, and
    let $S:=S(B_Z, f, \alpha)$ with $f:=(x^*, y^*)\in S_{Z^*}$ and $\alpha>0$. Then assume that $(x,0)\notin S$.
    We look at two cases:
    \begin{enumerate}
        \item[(a)] If $\|x^*\|=1$, then $x\notin S(B_X, x^\ast, \alpha)$. Thus there is a $u\in S(B_X, x^\ast, \alpha)$ such that $\|x-u\|\geq 2-\varepsilon$. Now $(u,0)\in S$ and $\|(x,0)-(u,0)\|_1\geq 2-\varepsilon$.
        \item[(b)] If $\|x^*\|<1$, then $\|y^*\|=1$. Find $y\in S_Y$ such that $y^*(y)>1-\alpha$. Then $(0,y)\in S$ and $\|(x,0)-(0,y)\|_1=\|x\|+\|y\|=2$, so we are done.
    \end{enumerate}

    Conversely, let us assume that $(x,0)$ is a $\nabla$-point in $Z$. Fix $\eps>0$ and let $S:=S(B_X, x^*, \alpha)$ with $x^*\in S_{X^*}$ and $\alpha>0$. Then assume that $x\notin S$. Let $\delta=\min\{\alpha, \varepsilon/2\}$. Then $(x,0)\notin S(B_Z, f, \delta)$ where $f:=(x^*,0)\in S_{Z^*}$. Since $(x,0)$ is a $\nabla$-point in $Z$, we can find $(u,v)\in S(B_Z, f, \delta)$ such that $\|x-u\|+\|v\|\geq 2-\delta$. Then $u\in S(B_X, x^*, \delta)\subset S(B_X, x^*, \alpha)$. In particular, $\norm{v}=1-\norm{u}<\delta$, hence $\|x-u\|\geq 2-2\delta\geq 2-\varepsilon$.
\end{proof}

\begin{proposition}
    Let $X$ and $Y$ be Banach spaces, let $x\in S_X$ and $y\in S_Y$, and let $a,b>0$ be such that $a+b=1$. If $(ax,by)$ is a $\nabla$-point in $X\oplus_1 Y$, then both $x$ and $y$ are Daugavet points. 
\end{proposition}

\begin{proof}
    Let $Z:=X\oplus_1 Y$ and assume that $(ax,by)$ is a $\nabla$-point in $Z$. Note that since $a,b>0$, we have that $(ax,by)$ is not extreme in $B_Z$ hence not strongly exposed. So by Theorem~\ref{thm:Nabla_not_str-exp_is_Daugavet}, $(ax,by)$ is actually a Daugavet point in $Z$. It then follows from \cite[Theorem~3.1]{HPV} that $x$ and $y$ are Daugavet points in $X$ and $Y$ respectively. 
\end{proof}

Finally, for $\ell_\infty$-sums, we have the following result. One direction is straightforward, and the other is analogous to \cite[Theorem~3.2]{HPV}. We leave the details to the reader.

\begin{proposition}

Let $X,Y$ be Banach spaces, and let $x\in B_X$ and $y\in B_Y$. Then $(x,y)$ is a $\nabla$-point in $X\oplus_\infty Y$ if and only if one of the two following conditions is satisfied:
\begin{enumerate}
    \item $x$ is a Daugavet point or $y$ is a Daugavet point; 

    \item $x$ and $y$ are both $\nabla$-points.
    
\end{enumerate}
    
\end{proposition}

\subsection{Function spaces}\label{subsection:function-spaces}

Let $(\Omega,\Sigma,\mu)$ be a measured space. It is well known that the space $L_1(\mu)$ has the Daugavet property if and only if $\mu$ admits no atom (see e.g. \cite[Section~2, Example~(b)]{Werner01}). In fact, it was observed in \cite[Proposition~4.12]{MPRZ} that this is actually equivalent to $L_1(\mu)$ having the strong diameter 2 property. Building on \cite[Theorem~3.1]{AHLP}, the following result was proved in \cite[Corollary~4.1]{MPRZ}.

\begin{proposition}

    Let $f\in S_{L_1(\mu)}$. Then the following assertions are equivalent:

    \begin{enumerate}
        \item $f$ is a super Daugavet point;

        \item $f$ is a $\Delta$-point;

        \item The support of $f$ contains no atom.
        
    \end{enumerate}
\end{proposition}

If $\mu$ admits atoms, then we can also show that $L_1(\mu)$  naturally contains $\nabla$-points which are not Daugavet points, and that those are exactly the points given by normalized indicator functions over an atom and their opposites, extending Example~\ref{expl:l1_nabla} in a natural way. 

\begin{proposition}

    Let $f\in S_{L_1(\mu)}$ and assume that the support of $f$ contains an atom $A$. Then $f$ is a $\nabla$-point if and only if $f=\pm \frac{\1_A}{\mu(A)}$.

\end{proposition}

\begin{proof}

If $\supp f$ contains an atom $A$, then either $f=g+\theta\alpha\1_A$ with $g\in L_1(\mu)$ non-zero, $\theta\in\{-1,1\}$ and $\alpha\in \left(0,\frac{1}{\mu(A)}\right)$; or $f=\frac{\theta}{\mu(A)}\1_A$ with $\theta\in\{-1,1\}$. In the first case, we have \begin{equation*}
    \norm{f-\frac{\theta}{\mu(A)}\1_A}=\int_{\Omega\backslash A}\abs{g} d\mu+(1-\alpha\mu(A))=2-2\alpha\mu(A)<2;
\end{equation*} so $f$ is at distance strictly less than $2$ to a denting point of $B_{L_1(\mu)}$ that is distinct from $f$, hence $f$ is not a $\nabla$-point by Proposition~\ref{prop:nabla_distance_2_to_denting}. In the second case, observe that $f=(0,1)$ in the Banach space $L_1(\mu_{\lvert \Omega\backslash A})\oplus_1 \R\equiv L_1(\mu)$, so Proposition~\ref{prop:nabla_points_l1_sums} yields that $f$ is a $\nabla$-point.
    
\end{proof}

Let $K$ be a compact Hausdorff space. It is well known that $C(K)$ has the Daugavet property if and only if $K$ has no isolated point (see e.g. \cite[Section~2, Example~(a)]{Werner01}). Building on \cite[Theorem~3.4]{AHLP}, the following result was proved in \cite[Corollary~4.3]{MPRZ}.

\begin{proposition}\label{prop:diametral_C(K)_characterization}

Let $f\in S_{C(K)}$. Then the following assertions are equivalent: 
 \begin{enumerate}
        \item $f$ is a super Daugavet point;

        \item $f$ is a $\Delta$-point;

        \item $f$ attains its norm at an accumulation point of $K$.
        
    \end{enumerate}

\end{proposition}

We have seen in Example~\ref{expl:l_infty^n_nabla} that, similar to $\ell_1^n$, the space $\ell_\infty^n$ contains plenty of $\nabla$-points. But unlike the infinite dimensional $L_1(\mu)$ setting, we will now prove that Proposition~\ref{prop:diametral_C(K)_characterization} also provides a characterization for $\nabla$-points in infinite dimensional $C(K)$-spaces, so that all the diametral notions coincide in this context.

\begin{proposition}

    Let $K$ be an infinite compact Hausdorff space. If a function $f\in S_{C(K)}$ does not attain its norm at an accumulation point of $K$, then it is not a $\nabla$-point.
    
\end{proposition}

\begin{proof}

Let us consider the set $H:=\left\{x\in K: \abs{f(x)}=1 \right\}$. Since $K$ is compact and $f$ does not attain its norm at an accumulation point of $K$, we have that the set $H$ is finite and in particular clopen in $K$. As a consequence, $\abs{f}$ attains its maximum on $K\backslash H$, and there exists $\eps\in(0,1]$ such that $\abs{f_{\lvert K\backslash H}}\leq 1-\eps$. Now fix any $x_0\in K\backslash H$, and for every $x\in H$, let $\theta_x:=\sign f(x)$. We consider the functional $\varphi:=\frac{1}{\abs{H}+1}\left(\sum_{x\in H} \theta_x\delta_x+\delta_{x_0}\right)\in S_{C(K)^*}$. We have $$\varphi(f)=\frac{\abs{H}+f(x_0)}{\abs{H}+1}\leq \frac{\abs{H}+1-\eps}{\abs{H}+1}=1-\frac{\eps}{\abs{H}+1},$$ so $f\notin S\left(\varphi,\frac{\eps}{\abs{H}+1}\right)$. Now pick any $g\in S\left(\varphi,\frac{\eps}{\abs{H}+1}\right)$ and pick some $z\in K$. If $z\notin H$, then $\abs{f(z)-g(z)}\leq \abs{f(z)}+\abs{g(z)}\leq 2-\eps$. Else, $$\theta_zg(z)=(\abs{H}+1)\left(\varphi(g)-\left(\varphi(g)-\frac{\theta_zg(z)}{\abs{H}+1}\right)\right)>\abs{H}+1-\eps-\abs{H}=1-\eps,$$ and as a consequence, $$\abs{f(z)-g(z)}=\abs{\theta_z-g(z)}=\abs{1-\theta_zg(z)}<\eps\leq 1.$$ Thus $\norm{f-g}\leq 2-\eps$, and $f$ is not a $\nabla$-point.
    
\end{proof}

\begin{remark}

{\ }

\begin{enumerate}
    \item Let $L$ be an infinite locally compact Hausdorff space. We can show analogously that if a function $f\in S_{C_0(L)}$ does not attain its norm at an accumulation point of $L$, then it is not a $\nabla$-point. In particular, if $L$ does not have an accumulation point, then $C_0(L)$ does not contain $\nabla$-points.

    \item If $X$ is an $L_1$-predual, then $X^{**}$ is a $C(K)$ space. It follows that all the diametral notions -- including the notion of $\nabla$-points -- coincide in $X$. A characterization for those points was provided in \cite[Theorem~3.2]{MRZ22}.
    
\end{enumerate}

\end{remark}

\subsection{Projective tensor products}\label{subsec:tensor_products}

The transfer of $\Delta$-points (respectively Daugavet points) in projective tensor products of Banach spaces was first investigated in \cite{LP2021} (respectively \cite{DJRZ2022}). We summarize the results obtained in these two papers here:

\begin{proposition}
Let $X$ and $Y$ be Banach spaces, and let $x_0\in S_X$ and $y_0\in S_Y$.
\begin{itemize}
    \item[(a)] If $x_0$ is a $\Delta$-point, then $x_0\otimes y$ is a $\Delta$-point in $X\widehat{\otimes}_\pi Y$ for every $y\in S_Y$ \cite[Remark~5.4]{LP2021}.
    \item[(b)] If $x_0\otimes y_0$ is a $\Delta$-point in $X\widehat{\otimes}_\pi Y$ and $y_0$ is a strongly exposed point, then $x_0$ is a $\Delta$-point \cite[Proposition~2.12, (a)]{DJRZ2022}.
    \item[(c)] If $x_0$ and $y_0$ are both Daugavet points, then $x_0\otimes y_0$ is a Daugavet point in $X\widehat{\otimes}_\pi Y$ \cite[Proposition~2.12, (b)]{DJRZ2022}.
    \item[(d)] If $x_0\otimes y_0$ is a Daugavet point in $X\widehat{\otimes}_\pi Y$ and $y_0$ is a denting point, then $x_0$ is a Daugavet point \cite[Proposition~2.12, (c)]{DJRZ2022}.
\end{itemize}
\end{proposition}

Our goal in this subsection is to study similar stability results for $\nabla$-points. In particular, we will show that in order to get that $x_0\otimes y_0$ is a Daugavet point in \cite[Proposition~2.12, (b)]{DJRZ2022}, it suffices to assume that one of the points is a Daugavet point and that the other is a $\nabla$-point (see Proposition~\ref{prop: Daugavet + Nabla = Daugavet}). 

We begin by pointing out the following simple lemma, which is certainly well known to experts on Daugavet points, but for which we could not find an explicit reference.

\begin{lemma}\label{lem: symmetric Daugavet point}
   Let $X$ be a Banach space. If $x\in S_X$ is a Daugavet point, then for every slice $S$ of $B_X$ and for every $\varepsilon>0$, there exists $y\in S$ such that $\|x\pm y\|\geq 2-\varepsilon$.
\end{lemma}

\begin{proof}
    Let $x$ be a Daugavet point, $S$ a slice of $B_X$, and $\varepsilon>0$. Since $x$ is a Daugavet point, then by \cite[Remark~2.3]{JRZ2022}, we can find a slice $\tilde{S}\subset S$ such that 
    \begin{equation*}
\|x-u\|\geq 2-\varepsilon\qquad\text{for all $u\in \tilde{S}$}.
    \end{equation*}
    Note that $-x$ is also a Daugavet point, hence we can find a $y\in \tilde{S}$ such that $\|-x-y\|=\|x+y\|\geq 2-\varepsilon$. Therefore, $\|x\pm y\|\geq 2-\varepsilon$ and $y\in S$ as we wanted.
\end{proof}

\begin{remark}
Let us point out that there is no complete $\nabla$-analogue to Lemma~\ref{lem: symmetric Daugavet point}. Indeed, $(0,1)$ is a $\nabla$-point in $\ell_1^2$, but if one considers any slice $S$ of $B_{\ell_1^2}$ such that $(0,-1)\in S$, but $(0,1), (1,0), (-1,0)\notin S$, then clearly there exists $\eps>0$ such that $\|(1,0)+ y\|\leq2-\eps$ for every $y\in S$. However, the result does hold true for a $\nabla$-point $x$ and slices that contain neither $x$ nor $-x$.

\end{remark}

\begin{proposition}\label{prop: Daugavet + Nabla = Daugavet}
Let $X$ and $Y$ be Banach spaces. If $x\in S_X$ is a Daugavet point and $y\in S_Y$ is a $\nabla$-point, then $x\otimes y\in S_{X\widehat{\otimes}_\pi Y}$ is a Daugavet point.
\end{proposition}
\begin{proof} We follow \cite[Proposition~2.12, (b)]{DJRZ2022}. Let $\varepsilon>0$ and $S:=S(B_{X\widehat{\otimes}_\pi Y}, B, \alpha)$ be an arbitrary slice. Our goal is to find a $z\in S$ such that $\|x\otimes y-z\|\geq 2-\varepsilon$. Find $x_0\in B_X$ and $y_0\in B_Y$ such that $B(x_0,y_0)>1-\alpha/2.$

    Consider first the following slice
    \begin{equation*}
S_1:=\left\{x\in B_X\colon B(x,y_0)>\sup_{u\in B_X} B(u,y_0)-\frac{\alpha}{4}\right\}.   
    \end{equation*}
    Since $x$ is a Daugavet point, by Lemma~\ref{lem: symmetric Daugavet point}, we can find $u_1\in S_1$ such that $\|x\pm u_1\|\geq 2-\varepsilon$.

    Now look at the slice
        \begin{equation*}
S_2:=\left\{y\in B_Y\colon B(u_1,y)>\sup_{v\in B_Y} B(u_1,v)-\frac{\alpha}{4}\right\}.   
    \end{equation*}

We consider two cases: (a) $y\in S_2$ and (b) $y\notin S_2$.
\begin{enumerate}
    \item[(a)] Assume that $y\in S_2$. We can then take $z:=u_1\otimes y$. Indeed, $\|x\otimes y- u_1\otimes y\|=\|x-u_1\|\cdot \|y\|\geq 2-\varepsilon$ and
    \begin{align*}
        B(u_1,y)&>\sup_{v\in B_Y} B(u_1,v)-\frac{\alpha}{4}\\
        &\geq B(u_1,y_0)-\frac{\alpha}{4}\\
        &>\sup_{u\in B_X} B(u,y_0)-\frac{\alpha}{4}-\frac{\alpha}{4}\\
        &\geq B(x_0,y_0)-\frac{\alpha}{2}\\
        &>1-\alpha.
    \end{align*}
    Hence, $z\in S$ and $\|x\otimes y-z\|\geq 2-\varepsilon$ as we wanted.
      \item[(b)] Assume that $y\notin S_2$. Then since $y$ is a $\nabla$-point we can find $v_2\in S_2$ such that $\|y-v_2\|\geq 2 -\varepsilon$. Similar computations as in (a) show that $B(u_1,v_2)>1-\alpha$, that is, $z:=u_1\otimes v_2\in S$. Finally, let us show that $\|x\otimes y - z\|\geq 2-\varepsilon$ also.

    Since $\|x+u_1\|\geq 2-\varepsilon$ and $\|y-v_2\|\geq 2-\varepsilon$, we can find $x^*\in S_{X^*}$ and $y^*\in S_{Y^*}$ such that 
      \begin{equation*}
      x^*(x+u_1)\geq 2-\varepsilon \qquad \text{and} \qquad  y^*(y-v_2)\geq 2-\varepsilon.
      \end{equation*}
      Therefore, $x^*(x),  x^*(u_1),  y^*(y), y^*(-v_2)\geq 1-\varepsilon$. Define a bilinear operator $B_0(x,y):=x^*(x)y^*(y)$ for all $x\in X$ and $y\in Y$. Then $\|B_0\|=1$ and
      \begin{align*}
          \|x\otimes y - u_1\otimes v_2\|&\geq B_0(x\otimes y - u_1\otimes v_2)\\
          &=x^*(x)y^*(y)-x^*(u_1)y^*(v_2)\\
          &=x^*(x)y^*(y)+x^*(u_1)y^*(-v_2)\\
          &\geq (1-\varepsilon)^2+(1-\varepsilon)^2=2(1-\varepsilon)^2.
      \end{align*}
\end{enumerate}
Hence, $x\otimes y$ is a $\nabla$-point.
\end{proof}



Although, we do not know whether an elementary tensor is a $\nabla$-point whenever both components are $\nabla$-points, we can at least say that such a point has to be far away from all denting points of the unit ball other than itself.

\begin{proposition}
   Let $X$ and $Y$ be Banach spaces. If $x\in S_X$ and $y\in S_Y$ are $\nabla$-points, then $\|x\otimes y -z\|=2$ for every denting point $z\in B_{X\widehat{\otimes}_\pi Y}$ with $z\neq x\otimes y$.
\end{proposition}

\begin{proof}
Let $z$ be a denting point in $B_{X\widehat{\otimes}_\pi Y}$ such that $z\neq x\otimes y$. By \cite[Corollary~4]{Werner1987}, we know that $z=u\otimes v$, where $u$ is a denting point of $B_X$ and $v$ is a denting point of $B_Y$. We will conclude the result arguing by cases:

\begin{enumerate}
    \item[(1)] Assume first that $x\neq u$ and $y=v$ (the case $x=u$ and $y\neq v$ is similar). Since $x$ is a $\nabla$-point, then by Proposition~\ref{prop:nabla_distance_2_to_denting}, we have that $\|x-u\|=2$. Hence, 
\[
\|x\otimes y -z\|=\|x\otimes y - u\otimes y\|=\|x-u\|\cdot \|y\|=2.
\]

\item[(2)] Assume now that  $x\neq u$ and $y\neq v$. 

\begin{enumerate}
    \item[(a)] Suppose $x=-u$ and $y=-v$. Then 
    \[
\|x\otimes y-u\otimes v\|= \|(-u)\otimes(-v)-u\otimes v\|=2\|u\|\|v\|=2.     
    \]
    \item[(b)] Suppose $x\neq -u$ and $y=-v$ (the case $x=-u$ and $y\neq -v$ is similar). Since $-u$ is also a denting point, we can use Propositon~\ref{prop:nabla_distance_2_to_denting} to obtain that $\|x+u\|=2$. Thus
       \[
\|x\otimes y-u\otimes v\|= \|x\otimes(-v)-u\otimes v\|=\|x+u\|\|v\|=2.     
    \] 
    \item[(c)]  Suppose $x\neq -u$ and $y\neq -v$. By using Propositon~\ref{prop:nabla_distance_2_to_denting} twice we obtain that 
\[
\|x+u\|=2 \quad \text{ and } \quad \|y-v\|=2. 
\]
Let $\varepsilon>0$. We can find $x^*\in S_{X^*}$ and $y^*\in S_{Y^*}$ such that 
      \begin{equation*}
      x^*(x+u)\geq 2-\varepsilon \qquad \text{and} \qquad  y^*(y-v)\geq 2-\varepsilon.
      \end{equation*}
      Therefore, $x^*(x),  x^*(u),  y^*(y), y^*(-v)\geq 1-\varepsilon$. Define a bilinear operator $B(r,s):=x^*(r)y^*(s)$ for all $r\in X$ and $s\in Y$. Then $\|B\|=1$ and
      \begin{align*}
          \|x\otimes y - u\otimes v\|&\geq B(x\otimes y - u\otimes v)\\
          &=x^*(x)y^*(y)-x^*(u)y^*(v)\\
          &=x^*(x)y^*(y)+x^*(u)y^*(-v)\\
          &\geq (1-\varepsilon)^2+(1-\varepsilon)^2=2(1-\varepsilon)^2.
      \end{align*}
\end{enumerate}

    \end{enumerate}
Hence, $\|x\otimes y -z\|=2$ as we wanted.
\end{proof}

\begin{remark}

    In particular, note that it follows from Proposition~\ref{prop:nabla-characterization_RNP} that if $X\widehat{\otimes}_\pi Y$ has the \RNP, then actually $x\otimes y$ is a $\nabla$-point in $X\widehat{\otimes}_\pi Y$ whenever $x$ and $y$ are $\nabla$-points. 
    
\end{remark}

We now turn our attention to the converse of Proposition~\ref{prop: Daugavet + Nabla = Daugavet}. For this we state a lemma that can be proven similarly to \cite[Remark~2.3]{JRZ2022} and is left to the reader.

\begin{lemma}\label{lem: a slice full of nabla-points}
Let $X$ be a Banach space, let $x\in S_X$ be a $\nabla$-point in $X$, and let $\varepsilon>0$. If $S$ is a slice of $B_X$ that does not contain $x$, then there exists a slice $\tilde{S}$ of $B_X$ contained in $S$ such that $\|x-y\|>2-\varepsilon$ for every $y\in \tilde{S}$.
\end{lemma}

Note that Lemma~\ref{lem: a slice full of nabla-points} gives us the following: If $x\otimes y\in S_{X\widehat{\otimes}_\pi Y}$ is a $\nabla$-point, $S$ is a slice of $B_{X\widehat{\otimes}_\pi Y}$ not containing $x\otimes y$, and $\varepsilon>0$, then we can find an elementary tensor $u\otimes v\in S$ such that $\|x\otimes y- u\otimes v\|\geq 2-\varepsilon$.

\begin{proposition}
    Let $X$ and $Y$ be Banach spaces. If  $x\otimes y\in S_{X\widehat{\otimes}_\pi Y}$ is a $\nabla$-point and $y$ is a denting point, then $x\in S_X$ is a $\nabla$-point.
\end{proposition}
\begin{proof} We follow \cite[Proposition~2.12, (c)]{DJRZ2022}. Let $S(B_X, x^\ast, \alpha)$ be such that $x\notin S(B_X, x^\ast, \alpha)$ and $\varepsilon>0$. Since $y$ is a denting point, we can find a slice $S(B_Y, y^*, \beta)$ such that $y\in S(B_Y, y^*, \beta)$ and $\text{diam}(S(B_Y, y^*, \beta))\leq \varepsilon$. 

    Define the bilinear form $B(u,v)=x^*(u)y^*(v)$ for all $u\in X$ and $v\in Y$. Consider the slice $S(B_{X\widehat{\otimes}_\pi Y}, B, \gamma)$, where $\gamma=\min\{\alpha,\beta\}$. Since $x\otimes y$ is a $\nabla$-point and $x\otimes y\notin S(B_{X\widehat{\otimes}_\pi Y}, B, \gamma)$, we can find, by Lemma~\ref{lem: a slice full of nabla-points}, an elementary tensor $u_0\otimes v_0\in S(B_{X\widehat{\otimes}_\pi Y}, B, \gamma)$ such that
    \begin{equation*}
    \|x\otimes y - u_0\otimes v_0\|\geq 2-\varepsilon.
    \end{equation*}
Observe that
\begin{align*}
    2-\varepsilon&\leq \|x\otimes y - u_0\otimes v_0\|\\
    &=\|(x-u_0)\otimes y + u_0\otimes (y-v_0)\|\\
    &\leq \|x-u_0\|+\|y-v_0\|.
\end{align*}
Since $u_0\otimes v_0\in S(B_{X\widehat{\otimes}_\pi Y}, B, \gamma)$, we have that 
\begin{equation*}
x^*(u_0)y^*(v_0)=B(u_0,v_0)>1-\gamma=1-\min\{\alpha,\beta\}.
\end{equation*}
Therefore, $x^*(u_0)>1-\alpha$ and $y^*(v_0)>1-\beta$, that is, $u_0\in S(B_X, x^\ast, \alpha)$ and $v_0\in S(B_Y, y^*, \beta)$. Note that $\|y-v_0\|<\varepsilon$, because $\text{diam}(S(B_Y, y^*, \beta))\leq \varepsilon$. Finally, 
\begin{equation*}
\|x-u_0\|\geq 2-\varepsilon - \|y-v_0\|>2-2\varepsilon,
\end{equation*}
hence $x$ is a $\nabla$-point as we wanted to show.
\end{proof}

We will end this section with some results on super $\Delta$-points, and with some observations on super Daugavet points. We first show that super $\Delta$-points pass easily to projective tensor products.

\begin{proposition}
  Let $X$ and $Y$ be Banach spaces. If $x$ is a super $\Delta$-point, then $x\otimes y\in S_{X\widehat{\otimes}_\pi Y}$ is a super $\Delta$-point for every $y\in S_Y$.  
\end{proposition}
\begin{proof}
 Let $y\in S_Y$. Since $x$ is a super $\Delta$-point, then by \cite[Proposition~3.4, (2)]{MPRZ}, there is a net $(x_\alpha)$ in $B_X$ that converges to $x$ weakly and such that $\|x_\alpha-x\|\rightarrow 2$. 

 Take $z_\alpha:=x_\alpha\otimes y$. Clearly, $(z_\alpha)\subset B_{X\widehat{\otimes}_\pi Y}$ and
 \begin{equation*}
 \|x\otimes y - x_\alpha\otimes y\|=\|x_\alpha-x\|\rightarrow 2.
 \end{equation*}
 Finally, let us show that $z_\alpha$ converges weakly to $x\otimes y$. Indeed, let $T\in (X\widehat{\otimes}_\pi Y)^*= \mathcal{L}(Y,X^*)$. Then
 \begin{equation*}
T(x_\alpha\otimes y)=\langle Ty, x_\alpha\rangle \rightarrow \langle Ty,x\rangle = T(x\otimes y),
 \end{equation*}
 because $x_\alpha$ converges to $x$ weakly, and the conclusion follows.
\end{proof}

We do not know whether super Daugavet points pass to the projective tensor product similarly to Daugavet points. However, let us point out that there is an analogue to Lemma~\ref{lem: symmetric Daugavet point} for those points.

\begin{lemma}\label{lem:symmetric_super-Daugavet point}

    Let $X$ be a Banach space, and let $x\in S_X$. If $x$ is a super Daugavet point, then for every relatively weakly open subset $V$ of $B_X$, and for every $\varepsilon>0$, there exists $y\in V$ such that $\|x\pm y\|\geq 2-\varepsilon$. In particular, we can find for every $y\in B_X$ a net $(y_a)$ in $ S_X$ that converges weakly to $y$ and such that $\norm{x\pm y_a}\to 2$.
    
\end{lemma}

\begin{proof}

    Let \begin{equation*}
        \Delta_\eps(x):=\{y\in B_X:\ \norm{x-y}>2-\eps\}.
    \end{equation*} It was observed in \cite[Proposition~3.4]{MPRZ} that the set $\Delta_\eps(x)$ is relatively weakly open in $B_X$. So as $x$ is a super Daugavet point, we get that $W:=V\cap \Delta_\eps(x)$ is a non-empty relatively weakly open subset of $B_X$. Thus $-W\cap\Delta_\eps(x)$ is also non-empty, and any element $y$ in this set satisfies $\norm{x\pm y}>2-\eps$. The desired net can then be constructed analogously to \cite[Proposition~3.4]{MPRZ}.
    
\end{proof}


\section{Renormings}\label{section:renormings}

Our starting point will be the observation that every Banach space can easily be renormed with a $\nabla$-point. Indeed, it immediately follows from Proposition~\ref{prop:nabla_points_l1_sums} and Example~\ref{expl:l1_nabla} that the addition of an ``$\ell_1$-corner" to the unit ball of any given Banach space automatically produces a $\nabla$-point. We will then combine this simple idea with an adapted version of the renorming from \cite[Theorem~2.4]{DKR+} that was already used in \cite{AALMPPV} to produce a renorming of $\ell_2$ with a super $\Delta$-point in the space and its dual.

\begin{proposition}\label{prop:nabla_renorming}

    Every Banach space can be renormed with a $\nabla$-point.
    
\end{proposition}

\begin{proof}

     For finite dimensional spaces, the statement is clear, as we have already seen in Example~\ref{expl:l1_nabla} that $\ell_1^n$ contains a $\nabla$-point for every $n\in \N$. So let $X$ be an infinite dimensional Banach space. Then $X$ is isomorphic to $Y\oplus_1 \R$ for any given co-dimension $1$ subspace $Y$ of $X$. Since $1$ is a $\nabla$-point in $\R$, we get from Proposition~\ref{prop:nabla_points_l1_sums} that $(0,1)$ is a $\nabla$-point in $Y\oplus_1 \R$, and the conclusion follows.  

  \end{proof}

Let $X$ be a Banach space with a Schauder basis $(e_n)$. Recall that $(e_n)$ is said to be \emph{unconditional} if the series $x:=\sum_{n\geq 1}a_ne_n$ converges unconditionally for every $x\in X$. Following \cite{AlbiacKalton} we will say that $(e_n)$ is \emph{1-unconditional} if for every $(a_n),(b_n)\in c_{00}$ with $\abs{a_n}\leq \abs{b_n}$ we have $$\norm{\sum_{n\geq 1}a_ne_n}\leq \norm{\sum_{n\geq 1}b_ne_n}.$$

The following theorem is the main result of the section. 

\begin{theorem}\label{thm:super_renorming}
    Let $X$ be an infinite dimensional Banach space with an unconditional weakly null Schauder basis $(e_n)$ and biorthogonal functionals $(e_n^*)$. Then there exists an equivalent
    norm $\nnorm{\cdot}$ on $X$ such that \begin{enumerate}
    
        \item\label{item:e1_super-Daugavet} $e_1$ is a super Daugavet point in $(X,\nnorm{\cdot})$;

        \item\label{item:e1*_weak*-super-Daugavet} $e_1^*$ is a weak$^*$ super Daugavet point in $(E,\nnorm{\cdot})$, where $E:=\cspan\{e_n^*\}$.

    \end{enumerate}

\end{theorem}

     We will prove Theorem~\ref{thm:super_renorming} in several steps. So from now on, $(X,\norm{\cdot})$ 
     will be an infinite dimensional Banach space with an unconditional weakly null Schauder basis $(e_n)$ and biorthogonal functionals $(e_n^*)$. We will also assume as we may that $(e_n)$ is normalized and 1-unconditional with respect to the original norm of $X$. For simplicity, we will keep using the notation from previous sections whenever we refer to this specific norm. 
     
     Let $Y:=\cspan\{e_n\}_{n\geq 2}$, and let $A$ be the set of all finitely supported elements in the positive cone of $Y$. Similarly,  let $F$ be the set of all finitely supported elements in the positive cone of $\cspan\{e_n^*\}_{n\geq 2}$. 
     
     We consider the equivalent norm $\nnorm{\cdot}$ on $X$ whose unit ball is 
     \begin{equation*}
         B_{(X,\nnorm{\cdot})}:=\cconv\{\pm(e_1+2x)\colon x\in A\cap B_X\}.
     \end{equation*} 
     Since $(e_n)$ is 1-unconditional, every finitely supported element $y\in B_Y$ can be written as $y=y_+-y_-$ with $y_+,y_-\in A\cap B_X$, and it immediately follows that $B_Y$ is contained in $B_{(X,\nnorm{\cdot})}$. Hence, \begin{equation*}
         B_{Y\oplus_1\R e_1}=\conv(B_Y\cup\{\pm e_1\})\subset B_{(X,\nnorm{\cdot})} \subset 3B_X,
     \end{equation*} and $\nnorm{\cdot}$ is indeed an equivalent norm on $X$. 

     We start by giving a geometric description of $B_{(E,\nnorm{\cdot})}$, and by producing useful formula for the norm $\nnorm{\cdot}$ and for its dual norm.

     \begin{lemma}\label{lem:properties_of_new_norm} The space $(X,\nnorm{\cdot})$ and the space $(E,\nnorm{\cdot})$ have the following properties:
     \begin{enumerate}
         \item \label{item:finite_supp_dual_ball} For every $z^*\in B_{(E,\nnorm{\cdot})}$ with finite support and $z^*(e_1)\ge0$, there exist $\lambda\in[0,1]$ and $x^*,y^*\in F\cap B_{X^*}$ with disjoint supports such that 
            \begin{equation*}z^*=\lambda(e_1^*-y^*)+(1-\lambda)\frac{1}{2}(x^*-y^*);\end{equation*} 
            
         \item \label{item:dual_ball_description}  The unit ball of the space $(E,\nnorm{\cdot})$ is given by
         \begin{equation*}B_{(E,\nnorm{\cdot})}=\cconv\Big\{\pm(e_1^*-x^*),\frac{1}{2}(x^*-y^*)\colon x^*,y^*\in F\cap B_{X^*}\Big\};\end{equation*}
         
         \item \label{item:norm_formula} For every $x\in X$,
         \begin{equation*}
             \nnorm{x}=\max\left\{\big|e_1^*(x)\big|,\big|e_1^*(x)-\|x_+\|\big|,\big|e_1^*(x)+\|x_-\|\big|,\frac{1}{2}\big(\|x_+\|+\|x_-\|\big)\right\},
         \end{equation*}
         where $x_+$ and $x_-$ are positive and negative parts of $x-e_1^*(x)e_1$ respectively; 
         
         \item \label{item:dual_norm_formula} For every $x^*\in E$,
         \begin{equation*}
             \nnorm{x^*}=\max\left\{\big|x^*(e_1)+2\|x_+^*\|\big|,\big|x^*(e_1)-2\|x_-^*\|\big|\right\},
         \end{equation*}
         where $x_+^*$ and $x_-^*$ are positive and negative parts of $x^*-x^*(e_1)e_1^*$ respectively.
     \end{enumerate}         
     \end{lemma}

\begin{proof}
    For convenience, we will first prove \ref{item:dual_norm_formula}, and then use it for proving \ref{item:finite_supp_dual_ball} and \ref{item:dual_ball_description}. Fix $x^*\in E$ and let $x_+^*$ and $x_-^*$ be the positive and negative parts of $x^*-x^*(e_1)e_1^*$ respectively. Then $x^*=x^*(e_1)e_1^*+x_+^*-x_-^*$. Since
    \begin{equation*}
         B_{(X,\nnorm{\cdot})}=\cconv\{\pm(e_1+2x)\colon x\in A\cap B_X\},
    \end{equation*}
    we get 
    \begin{align*}
        \nnorm{x^*}&=\sup_{x\in A\cap B_X}\big|x^*(e_1+2x)\big|\\
        &=\sup_{x\in A\cap B_X}\big|x^*(e_1)+2x_+^*(x)-2x_-^*(x)\big|\\
        &=\max\left\{\big|x^*(e_1)+2\|x_+^*\|\big|,\big|x^*(e_1)-2\|x_-^*\|\big|\right\}.
    \end{align*}
    The last equality strongly relies on the fact that $(e_n)$ is 1-unconditional with respect to the original norm. Indeed, the latter implies that 
    \begin{equation*}\norm{y^*}=\sup\{y^*(x): x\in A\cap B_X,\ \supp{x}\subset\supp{y^*}\}\end{equation*}
    for every functional $y^*\in F$. So since $x^*_+$ and $x^*_-$ have disjoint supports, each of these functionals admits a norming set in $A\cap B_X$ on which the other functional vanishes.

    Now let us prove \ref{item:finite_supp_dual_ball}. Fix $z^*\in B_{(E,\nnorm{\cdot})}$ with finite support and $z^*(e_1)\ge0$.
    Let $z_+^*$ and $z_-^*$ be the positive and negative parts of $z^*-z^*(e_1)e_1^*$ respectively. If $z^*(e_1)=1$, then from \ref{item:dual_norm_formula} we get
    \begin{equation*}\max\left\{\big|1+2\|z_+^*\|\big|,\big|1-2\|z_-^*\|\big|\right\}=\nnorm{z^*}\le 1,\end{equation*}
    which means $z_+^*=0$ and $\|z_-^*\|\le 1$. Hence $z^*=e_1^*-z_-^*$ and $z_-^*\in  F\cap B_{X^*}$.
    
    So let us assume that $z^*(e_1)\in[0,1)$. Then
    \begin{equation*}z^*=z^*(e_1)e_1^*+z_+^*-z_-^*=z^*(e_1)\left(e_1^*-y^*\right)+\big(1-z^*(e_1)\big)\frac{1}{2}\left(x^*-y^*\right),\end{equation*}
    where $x^*=\frac{2}{1-z^*(e_1)}z_+^*$ and $y^*=\frac{2}{1+z^*(e_1)} z_-^*$.
    Furthermore,
    \begin{equation*}\max\left\{\big|z^*(e_1)+2\|z_+^*\|\big|,\big|z^*(e_1)-2\|z_-^*\|\big|\right\}=\nnorm{z^*}\le 1\end{equation*}
    and thus 
    \begin{equation*}\frac{2\|z_+^*\|}{1-z^*(e_1)}\le 1\quad\text{ and }\quad \frac{2\|z_-^*\|}{1+z^*(e_1)}\le 1 ,\end{equation*}
    meaning $x^*,y^*\in F\cap B_{X^*}$.

    Next we will prove \ref{item:dual_ball_description}. By \ref{item:finite_supp_dual_ball} we have 
    \begin{equation*}B_{(E,\nnorm{\cdot})}\subseteq\cconv\Big\{\pm(e_1^*-x^*),\frac{1}{2}(x^*-y^*)\colon x^*,y^*\in F\cap B_{X^*}\Big\}.\end{equation*}
    From \ref{item:dual_norm_formula} we get
    \begin{equation*}\nnorm{e_1^*-x^*}=\max\left\{\big|1+2\|0\|\big|,\big|1-2\|x^*\|\big|\right\}=1\end{equation*}
    and 
    \begin{equation*}\frac{1}{2}\nnorm{x^*-y^*}=\max\left\{\|x^*\|,\|y^*\|\right\}\le 1\end{equation*}
    for all $x^*,y^*\in F\cap B_{X^*}$, which gives us the other inclusion. 

    Finally, let us prove \ref{item:norm_formula}. For this purpose, we start by proving that the basis $(e_n)$ is also monotone with respect to the new norm. Then we will get e.g. from \cite[Lemma~3.2.3]{AlbiacKalton}  that \begin{equation*}
        \nnorm{x}=\sup_{z^*\in B_{(E,\nnorm{\cdot})}} \big|z^*(x)\big|
    \end{equation*} for every $x\in X$. So let $n\in \N$ and let $P_n$ be the projection on $\spn\{e_1,\dots, e_n\}$. Since $(e_n)$ is 1-unconditional hence monotone with respect to the original norm, we have that $P_n(x)\in A\cap B_X$ for every $x\in A\cap B_X$. Thus $P_n(e_1+2x)=e_1+2P_n(x)\in B_{(X,\nnorm{\cdot})}$, and from this we clearly get $P_n\big(B_{(X,\nnorm{\cdot})}\big)\subseteq B_{(X,\nnorm{\cdot})}$. Therefore $\nnorm{P_n}\le 1$, which is what we wanted. 
    
    Now fix $x\in X$ and let  $x_+$ and $x_-$ be the positive and negative parts of $x-e_1^*(x)e_1$ respectively. By \ref{item:dual_ball_description} we have
    \begin{align*}
        \nnorm{x}&=\sup_{z^*\in B_{(E,\nnorm{\cdot})}}\big|z^*(x)\big|\\
        &= \max\left\{\sup_{x^*\in F\cap B_{X^*}}|(e_1^*-x^*)(x)|,\frac{1}{2}\sup_{x^*,y^*\in F\cap B_{X^*}}|(x^*-y^*)(x)|\right\}\\
        &= \max\left\{\sup_{x^*\in F\cap B_{X^*}}\big|e_1^*(x)-x^*(x_+)+x^*(x_-)\big|,\frac{1}{2}\sup_{x^*,y^*\in F\cap B_{X^*}}\big(x^*(x_+)+y^*(x_-)\big)\right\}\\
        &= \max\left\{\big|e_1^*(x)\big|,\big|e_1^*(x)-\|x_+\|\big|,\big|e_1^*(x)+\|x_-\|\big|,\frac{1}{2}\big(\|x_+\|+\|x_-\|\big)\right\}.
    \end{align*}

    Again, the last two equalities strongly rely on the fact that $(e_n)$ is 1-unconditional with respect to the original norm because the latter implies that every $y\in A$ attains its norm on the set \begin{equation*}\{x^*\in F\cap B_{X^*}:\ \supp{x^*}\subset\supp{y}\},  \end{equation*} and in particular that $x_+$ and $x_-$ both admit a norming functional in $F\cap B_{X^*}$ which vanishes at the other point.
\end{proof}

In particular, note that $\nnorm{x}=1$ for every $x\in A\cap S_X$, and that $\nnorm{e_1+2y}=1$ for every $y\in A\cap B_X$.  Furthermore, $S_{(X,\nnorm{\cdot})}$ contains plenty of line segments that will prove very useful later on.

\begin{corollary}\label{cor:renorm_segments_in_sphere}
    For every $x\in A\cap S_X$ and $y\in A\cap B_X$ with disjoint supports, and for every $\lambda\in[0,1]$, we have \begin{equation*}
        \nnorm{\lambda x-(1-\lambda)(e_1+2y)}=1.
    \end{equation*}
\end{corollary}

\begin{proof}
    Fix $x\in A\cap S_{X}$ and $y\in A\cap B_{X}$ with disjoint supports, and fix $\lambda\in[0,1]$. Then by Lemma \ref{lem:properties_of_new_norm}~\ref{item:norm_formula} we have 
    \begin{equation*}
        \nnorm{\lambda x-(1-\lambda)(e_1+2y)}=\max\left\{1-\lambda,1,1-\lambda,\frac{1}{2}\big(\lambda+2(1-\lambda)\big)\right\}=1.
    \end{equation*}
\end{proof}

    Analogously, observe that $\nnorm{x^*}=2$ for every $x^*\in F\cap S_{X^*}$ and that $\nnorm{e_1^*-x^*}=1$ for every $x^*\in F\cap B_{X^*}$. From these observations, one can already easily deduce some diametral properties for the points $e_1$ and $e_1^*$. So for the reader's convenience, let us prove right away the two following corollaries. 

     \begin{corollary}\label{cor:e1_super-Delta}

         The point $e_1$ is a super $\Delta$-point in  $(X,\nnorm{\cdot})$, and $e_1^*$ is a weak$^*$ super $\Delta$-point in $(E,\nnorm{\cdot})$.
         
     \end{corollary}

     \begin{corollary}\label{cor:e1_nabla}

         The points $e_1$ and $e_1^*$ are $\nabla$-points in $(X,\nnorm{\cdot})$ and $(E,\nnorm{\cdot})$ respectively. 
         
     \end{corollary}

     \begin{proof}[Proof of Corollary~\ref{cor:e1_super-Delta}]

      Since $(e_n)$ is weakly null, we have $e_1+2e_n\to e_1$ weakly. By Corollary~\ref{cor:renorm_segments_in_sphere}, we have \begin{equation*}
     \nnorm{e_1}=\nnorm{e_1+2e_n}=1\ \text{and}\ \nnorm{e_1-(e_1+2e_n)}=2\nnorm{e_n}=2
     \end{equation*}for every $n\geq 2$. Hence $e_1$ is a super $\Delta$-point in $(X,\nnorm{\cdot})$. 

     Analogously, $(e_n^*)$ is weak$^*$ null, so we have $e_1^*-e_n^*\to e_1^*$ weak$^*$. Then as $\nnorm{e_1^*-(e_1^*-e_n^*)}=\nnorm{e_n^*}=2$ for every $n\geq2$, it follows that $e_1^*$ is a weak$^*$ super $\Delta$-point in $(E,\nnorm{\cdot})$. 
     
      \end{proof}

      \begin{proof}[Proof of Corollary~\ref{cor:e1_nabla}]

      Clearly, $B_{(X,\nnorm{\cdot})}$ is equal to the closure of the convex hull of the set \begin{equation*}
          \{\pm e_1\}\cup\{\pm (e_1+2x)\colon x\in A\cap S_X\}.
      \end{equation*} So every slice of $B_{(X,\nnorm{\cdot})}$ contains either $\pm e_1$, or $\pm (e_1+2x)$ for some $x\in A\cap S_X$. By Corollary~\ref{cor:renorm_segments_in_sphere}, we have
      \begin{equation*}
          \nnorm{e_1-(e_1+2x)}=2\nnorm{x}=2\ \text{and}\ \nnorm{e_1+(e_1+2x)}=2\nnorm{e_1+x}=2
      \end{equation*} for every $x\in A\cap S_X$, and it clearly follows that $e_1$ is a $\nabla$-point in $(X,\nnorm{\cdot})$.

      It is also quite clear from the proof of Lemma~\ref{lem:properties_of_new_norm} that $B_{(E,\nnorm{\cdot})}$ is equal to the closure of the convex hull of the set \begin{equation*}
          \{\pm e_1^*\}\cup\{\pm (e_1^*-x^*)\colon x^*\in S_{X^*}\}\cup \left\{\frac{1}{2}(x^*-y^*):\  x^*,y^*\in F\cap S_{X^*},\ \supp{x^*}\cap\supp{y^*}=\emptyset\right\}.\end{equation*}
    Using Lemma~\ref{lem:properties_of_new_norm}~\ref{item:dual_norm_formula}, one can easily check that \begin{equation*}
        \nnorm{e_1^*\pm(e_1^*-x^*)}=\nnorm{e_1^*-\frac{1}{2}(x^*-y^*)}=2
    \end{equation*} for every $x^*,y^*\in F\cap S_{X^*}$ with disjoint supports, so it follows as above that $e_1^*$ is a $\nabla$-point in $(E,\nnorm{\cdot})$.

      \end{proof}

      In particular, it readily follows from Corollaries~\ref{cor:e1_super-Delta} and \ref{cor:e1_nabla} that $e_1$ is a Daugavet point in $(X,\nnorm{\cdot})$ and that $e_1^*$ is a weak$^*$ Daugavet point in $(E,\nnorm{.})$. We will now finally show that those points are a super Daugavet point and a weak$^*$ super Daugavet point in their respective unit balls.

\begin{proof}[Proof of Theorem~\ref{thm:super_renorming}]
    Let us first prove \ref{item:e1_super-Daugavet}. Fix $z\in B_{(X,\nnorm{\cdot})}$ with finite support. We will construct a sequence $(w_n)$ in $B_{(X,\nnorm{\cdot})}$ such that $\|e_1- w_n\|\rightarrow2$ and $(w_n)$ converges weakly to $z$. 
    
 Let  $z_+$ and $z_-$ be the positive and negative parts of $z-e_1^*(z)e_1$ respectively, and let $\lambda:=\frac{1+e_1^*(z)}{2}\in [0,1]$. Then define \begin{equation*}
     x:=\left\{\begin{array}{cc}
        \frac{1}{2\lambda}z_+  & \text{if $\lambda>0$}, \\
        0  & \text{if $\lambda=0$},
     \end{array}\right.\quad \text{and}\quad y:=\left\{\begin{array}{cc}
        y=\frac{1}{2(1-\lambda)}z_-  & \text{if $\lambda<1$}, \\
        0  & \text{if $\lambda=1$}.
     \end{array}\right.
 \end{equation*}Note that if $\lambda=0$, then $e_1^*(z)=-1$, and it follows from Lemma \ref{lem:properties_of_new_norm}~\ref{item:norm_formula} that $z_+=0$. Analogously, if $\lambda=1$, then $z_-=0$.    So in either case, we have
    \begin{equation*}z=\lambda(e_1+2x)-(1-\lambda)(e_1+2y).\end{equation*}
    By Lemma \ref{lem:properties_of_new_norm}~\ref{item:norm_formula} 
    we have 
    $\|z_+\|-e_1^*(z)\le \nnorm{z}\le 1$ and thus
    \begin{equation*}\|x\|=\frac{1}{2\lambda}\|z_+\|\le \frac{1+e_1^*(z)}{2\lambda}=1,\end{equation*}
    if $\lambda\neq0$.
    So for every $n\in \N$ large enough, we have $\|x\|\le1$ and $\|x+e_{n}\|\ge \norm{e_n}= 1$. Thus there exists $a_n\in [0,1]$ such that $x+a_ne_{n}\in S_{X}$. Let $w_n=z+2\lambda a_ne_{n}$. By construction, $w_n\in B_{(X,\nnorm{\cdot})}$, and since $(e_n)$ is weakly null, $w_n\to z$ weakly. Furthermore, $x+a_ne_{n}\in A\cap S_{X}$,  and $\supp(x+a_ne_{n})\cap \supp{y}=\emptyset$ for $n\in\N$ large enough, so by Corollary~\ref{cor:renorm_segments_in_sphere}
    we get
    \begin{equation*}
     \nnorm{e_1-w_n}=2\nnorm{(1-\lambda)(e_1+y)-\lambda (x+a_ne_{n})}=2
     \end{equation*}as desired. Finally, as the set of all finitely supported elements is dense in $B_{(X,\nnorm{\cdot})}$, then this  immediately implies that for every $w\in B_{(X,\nnorm{\cdot})}$, there exists a net $(w_\alpha)$ in $B_{(X,\nnorm{\cdot})}$ that converges weakly to $w$ and satisfies $\nnorm{e_1-w_\alpha}\rightarrow2$. Hence $e_1$ is a super Daugavet point in $(X,\nnorm{\cdot})$. 

     Next let us prove \ref{item:e1*_weak*-super-Daugavet}.  Fix $z^*\in B_{(E,\nnorm{\cdot})}$ with finite support. We will construct as above a sequence $(w_n^*)$ in $B_{(E,\nnorm{\cdot})}$ such that $\|e_1^*- w_n^*\|\rightarrow2$ and $(w_n^*)$ converges weak$^*$ to $z^*$. Since the set of all finitely supported functionals is dense in $B_{(E,\nnorm{\cdot})}$, this will give us that $e_1^*$ is a weak$^*$ super Daugavet point in $(E,\nnorm{\cdot})$.
     
     Actually, we will show  that $\|e_1^*\pm w_n^*\|\rightarrow2$, which allows us to assume that $z^*(e_1)\ge 0$. 
     Then by Lemma \ref{lem:properties_of_new_norm}~\ref{item:finite_supp_dual_ball}, there exist $\lambda\in[0,1]$ and $x^*,y^*\in F\cap B_{X^*}$ such that 
    \begin{equation*}z^*=\lambda(e_1^*-y^*)+(1-\lambda)\frac{1}{2}(x^*-y^*).\end{equation*}
    Again, $\|x^*\|\le1$ and $\|x^*+e_{2n}^*\|\ge \norm{e_{2n}^*}= 1$ for $n\in\N$ large enough. Thus there exists $a_n\in [0,1]$ such that $x^*+a_ne_{2n}^*\in S_{X^*}$. Similarly, there exists $b_n\in [0,1]$ such that $y^*+b_ne_{2n+1}^*\in S_{X^*}$. Let \begin{equation*}
        w_n^*=z^*+\frac{1-\lambda}{2} a_ne_{2n}^*-\frac{1+\lambda}{2} b_ne_{2n+1}^*.
    \end{equation*} Since $(e_n^*)$ is weak$^*$ null, $w_n^*\to z^*$ weak$^*$. Furthermore, $\supp(x^*+a_ne_{2n}^*)\cap \supp(y^*+b_ne_{2n+1}^*)=\emptyset$ for $n\in\N$ large enough, so by Lemma \ref{lem:properties_of_new_norm}~\ref{item:dual_norm_formula} 
    we get
    \begin{equation*}
     \nnorm{e_1^*-w_n^*}=\nnorm{(1-\lambda)e_1^*-\frac{1-\lambda}{2}(x^*+a_ne_{2n}^*)+\frac{1+\lambda}{2}(y^*+b_ne_{2n+1}^*)}=\max\{2,0\}=2
     \end{equation*}
     and
    \begin{equation*}
     \nnorm{e_1^*+w_n^*}=\nnorm{(1+\lambda)e_1^*+\frac{1-\lambda}{2}(x^*+a_ne_{2n}^*)-\frac{1+\lambda}{2}(y^*+b_ne_{2n+1}^*)}=\max\{2,0\}=2.
     \end{equation*}The conclusion follows.
\end{proof}

\begin{remark}\label{rem:super_renorming}
  {\ }
  
\begin{enumerate}

    \item\label{item:ccw_rem} It is straightforward to check that the point $e_1$ is an extreme point of $B_{(X,\nnorm{\cdot})}$. So it actually follows from \cite[Remark~3.14]{MPRZ} that additionally to being a super Daugavet point, $e_1$ is also a ccw $\Delta$-point (i.e. satisfies a $\Delta$-like condition for convex combinations of relative weakly open subsets of the unit ball, see \cite{MPRZ}).
    Similarly, $e_1^*$ is an extreme point of $B_{(E,\nnorm{\cdot})}$, so additionally to being a weak$^*$ super Daugavet point, it is also a weak$^*$ ccw $\Delta$-point. In strongly regular spaces (respectively weak$^*$ strongly regular duals), this is the best diametral property we can ask for points of the unit ball, as the existence of a (weak$^*$) ccw Daugavet point implies the (weak$^*$) SD2P (see \cite[Proposition~3.12]{MPRZ}).

    \item\label{item:sequential_rem} It is actually possible to prove, up to some technical and longish adjustments in the above proof (that rely once again strongly on unconditionality), that for every $w\in B_{(X,\nnorm{\cdot})}$, there exists a sequence $(w_n)$ in $B_{(X,\nnorm{\cdot})}$ that converges weakly to $z$ and such that $\nnorm{e_1-w_n}\to 2$. Hence $e_1$ is actually a \emph{sequential super Daugavet point} in the lines of \cite[Definition~5.22]{AALMPPV}. Recall that in general, this is strictly stronger than merely being a super Daugavet point, as there exists a Banach space with the Daugavet property and the Schur property \cite{KW_Schur+Daugavet}. Also, the sequence $(w_n)$ could easily be modified (as in the second part of the above proof) so that $\nnorm{e_1\pm w_n}=2$ for every $n\in \N$, so $e_1$ satisfies a sequential symmetric super Daugavet condition similar to the one from Lemma~\ref{lem:symmetric_super-Daugavet point}. It is not clear whether those two sequential properties are equivalent in general. 

\end{enumerate}
     
\end{remark}

It is well known that every shrinking basis is weakly null (a.k.a. \emph{semi-shrinking}) and that every unconditional basis in a space that does not contain a copy of $\ell_1$ is shrinking. So we immediately get the following corollary. On the other hand, let us note that there exists bases which are semi-shrinking but not shrinking (e.g. the Faber--Schauder basis of $C[0,1]$), and actually that a continuum of mutually non-similar such bases can be constructed using tensor products \cite{Holub71}. Furthermore, there exists a weakly null sequence with no shrinking subsequence \cite{Odell80}. In the previous examples, the bases are all conditional, but unconditional semi-shrinking bases which are not shrinking do also exist (see e.g. \cite{PS} or \cite[Examples~2]{DD}). 

 \begin{corollary}\label{cor:shrinking_super_renorming}

     Let $X$ be a infinite dimensional Banach space with an unconditional Schauder basis. If $X$ does not contain a copy of $\ell_1$, and in particular if $X$ is reflexive, then $X$ can be renormed so that $X$ contains a super Daugavet point and  $X^*$ contains a weak$^*$ super Daugavet point.

 \end{corollary}

\begin{remark}\label{rem:weakly-null_biorthogonal-functionals}

    If the sequence $(e_n^*)$ of biorthogonal functionals is also assumed to be weakly null in Theorem~\ref{thm:super_renorming}, and in particular if the space $X$ is reflexive, then the functional $e_1^*$ is actually also a (symmetric sequential) super Daugavet point in $(E,\nnorm{\cdot})$. Note that the basis from \cite[Examples~2]{DD} is unconditonal, boundedly complete and semi-shrinking; but as it is not shrinking, the space is not reflexive, so there exits non-trivial examples of this kind. Recall that in general, diametral properties of points in a dual space are way stronger than their weak$^*$ counterparts, as e.g. every point in the unit ball of $C[0,1]^*$ is a weak$^*$ ccw Daugavet point while $B_{C[0,1]^*}$ contains denting points. 

\end{remark}

In particular, let us highlight that combining Corollary~\ref{cor:shrinking_super_renorming} with Remarks~\ref{rem:super_renorming}~\ref{item:ccw_rem} and \ref{rem:weakly-null_biorthogonal-functionals}, we get the following theorem. 

\begin{theorem}\label{thm:l2_renorming}

    Let $(e_n)$ be the unit vector basis of $\ell_2$. There exists a renorming of $\ell_2$ for which $e_1$ is an extreme super Daugavet -- hence ccw $\Delta$ -- point in the new norm and its dual norm. 
    
\end{theorem}

 In \cite[Example~5.20]{AALMPPV}, the renorming of $\ell_2$ with a super $\Delta$-point from \cite[Theorem~3.1]{AALMPPV} was used to provide a non-reflexive M-embedded space with a super $\Delta$-point and a super $\Delta$-point in its dual. By using Theorem~\ref{thm:l2_renorming}, we can produce a similar construction for super Daugavet point.

\begin{corollary}\label{cor:M-embedded_renormings}

There exists a non-reflexive M-embedded space $Y$ such that $Y$ and its dual contain a super Daugavet point.

\end{corollary}

\begin{proof}

   Let $X:=(\ell_2,\nnorm{\cdot})$ be the renorming of $\ell_2$ for which $e_1$ is a super Daugavet point in $X$ and $X^*$. As reflexive spaces are trivially M-embedded, it follows from  \cite[Theorem~III.1.6]{HWW} that the space $Y:=c_0(X)$ is M-embedded.  Then $Y^*\equiv\ell_1(X^*)$, and it follows from \cite[Remark~3.28]{MPRZ} that the point $(e_1,0,0,\dots)$ is a super Daugavet point in $Y$ and $Y^*$. 
   
\end{proof}

It was proved in  \cite[Theorem~4.1]{AALMPPV} that every infinite dimensional Banach spaces can be renormed to have a $\Delta$-point. It is thus natural to ask the following.

\begin{question}

Can every infinite dimensional Banach space be renormed to have a Daugavet point?

\end{question} 

By Theorem~\ref{thm:super_renorming} and by \cite[Theorem~2.1]{AALMPPV}, the answer is yes if $X$ is infinite dimensional and has a weakly null unconditional Schauder basis, or if $X:=\ell_1$. Combining those two results, we will now prove that more generally, the answer is yes for every infinite dimensional Banach space with an unconditional Schauder basis, as well as for any Banach space that contains a complemented copy of such a space.

Note that some of the key ingredients for the $\Delta$-renormings from \cite[Theorem~4.1]{AALMPPV} were a classic norm extension result and the well known fact that $\Delta$-points pass to superspaces. This is no longer true for Daugavet points, e.g. because it was proved in  \cite{AHLP} that $\ell_2$-sums of Banach spaces never contain such points. However, let us point out that we can still get, up to renorming, some similar result whenever the considered point lives in a space that is complemented in the superspace.

\begin{proposition}\label{prop:DP_in_complemented_subspaces}

    Let $X$ be a Banach space. If $X$ contains a complemented subspace $Y$ that can be renormed with a Daugavet point, then $X$ can be renormed with a Daugavet point.  Moreover, if $Y$ can  be renormed so that $Y^*$ contains a (weak$^*$) Daugavet point, then so can $X$. The same does hold for (weak$^*$) super Daugavet points.
    
\end{proposition}

\begin{proof}

As $Y$ is complemented in $X$, we have that $X$ is isomorphic to $Y\oplus_1 Z$ for some Banach space $Z$. As this space is also isomorphic to $(Y,\nnorm{\cdot})\oplus_1Z$ for every equivalent norm $\nnorm{\cdot}$ on $Y$, we may simply assume that $Y$ contains a Daugavet point. But now if $y\in S_Y$ is a Daugavet point, then it follows from \cite[Proposition~2.3]{HPV} that $x:=(y,0)$ is a Daugavet point in $Y\oplus_1 Z$. For super Daugavet points, the result follows analogously using \cite[Remark~3.28]{MPRZ}. Furthermore, as the dual of this space is isometric to $(Y^*,\nnorm{\cdot})\oplus_\infty Z^*$, then the dual part of Proposition~\ref{prop:DP_in_complemented_subspaces} follows from known transfer results of Daugavet and super Daugavet points through $\ell_\infty$-sums (see the following remark).
    
\end{proof}

\begin{remark}

    If we had initially taken an $\ell_\infty$-sum instead of an $\ell_1$-sum in the proof of the previous result, then observe that we would actually get a renorming of $X$ with infinitely many (super) Daugavet points. Indeed, it follows from \cite[Proposition~2.4]{HPV} that if $y\in S_Y$ is a Daugavet point, then $(y,z)$ is a Daugavet point in $Y\oplus_\infty Z$ for every $z\in B_Z$. For super Daugavet points, this follows again from \cite[Remark~3.28]{MPRZ}.

\end{remark}

As a corollary, we get that every Banach space with an unconditional basis can be renormed with a Daugavet point and a weak$^*$ Daugavet point in its dual. 

\begin{corollary}\label{cor:unconditional_renorming}

    Let $X$ be a Banach space with an unconditional basis $(e_n)$. Then $X$ can be renormed with a Daugavet point. More precisely:

    \begin{enumerate}
    
        \item If $(e_n)$ is shrinking, or if $(e_n)$ is neither shrinking nor boundedly complete, then $X$ can be renormed with a super Daugavet point and a weak$^*$ super Daugavet point in its dual;

        \item If $(e_n)$ is not shrinking, then $X$ can be renormed with a Daugavet point in the space and its dual.

    \end{enumerate}
 
 \end{corollary}

 \begin{proof}

     If $(e_n)$ is shrinking, then this is Corollary~\ref{cor:shrinking_super_renorming}. The remaining cases are direct consequences of Proposition~\ref{prop:DP_in_complemented_subspaces} together with classic results from James about unconditional bases: If $(e_n)$ is not boundedly complete, then $X$
     contains a complemented copy of $c_0$, and if $(e_n)$ is not shrinking, then $X$  contains a complemented copy of $\ell_1$ (see e.g. \cite[Theorems~3.3.2 and 3.3.1]{AlbiacKalton}). The fact that $\ell_1$ can be renormed with a Daugavet point in the space and its dual follows by combining \cite[Theorem~2.1]{AALMPPV} and \cite[Proposition~5.1]{VeeorgFunc}.

     \end{proof}

Using once again Proposition~\ref{prop:DP_in_complemented_subspaces}, we ultimately get the result that was stated in the introduction (Theorem~\ref{intro-thm:unconditional_renorming}). 

Finally, observe that calling to some other classic results from the literature, we also have the following statements.

\begin{corollary}\label{cor:complemented_renormings} Let $X$ be a Banach space. 
    
\begin{enumerate}

     \item\label{item:separabel_c0} If $X$ is separable and contains a copy of $c_0$, then $X$ can be renormed with a super Daugavet point and a weak$^*$ super Daugavet point in its dual.

     \item\label{item:dual_c0} If $X^*$ contains a copy of $c_0$, then $X$ can be renormed with a Daugavet point in the space and its dual. 

    \item\label{item:l_infty} If $X$ contains a copy of $\ell_\infty$, then $X$ can be renormed with a super Daugavet point in the space and its dual.  
    
\end{enumerate}

\end{corollary}

\begin{proof}

    It was proved by Sobczyk in \cite{Sobczyk41} that copies of $c_0$ in separable Banach spaces are always complemented. In the same paper, it was also observed that a previous extension result from Phillips can be used to show that the same is true for copies of $\ell_\infty$ in arbitrary Banach spaces. With modern terminology, this is a consequence of the fact that the space $c_0$ is separably injective, and that the space $\ell_\infty$ is isometrically injective. We refer to \cite[Section~2.5]{AlbiacKalton} for more details.
    It is also a well known result, due to Bessaga and Pe{\l}czy{\'n}ski, that if the dual $X^*$ of a Banach space $X$ contains a copy of $c_0$, then the space $X$ contains a complemented copy of $\ell_1$ (see e.g. \cite[Theorem~4.4]{Czechbook}).

    With those results at hand, Corollary~\ref{cor:complemented_renormings} immediately follows from Proposition~\ref{prop:DP_in_complemented_subspaces} together with Corollary~\ref{cor:unconditional_renorming} and the well known fact that $\ell_\infty$, as a $C(K)$-space, admits super Daugavet points (combining e.g. \cite[Corollary~5.4]{AHLP} and \cite[Corollary~4.3]{MPRZ}). That $\ell_\infty^*$ also admits super Daugavet points can be obtained as follows. First, it is well known that $\ell_\infty^*$ contains an isometric copy of the space $L_1[0,1]$. Indeed, $\ell_\infty$ contains an isometric copy of $\ell_1$, so the latter follows e.g. from the results from \cite{DGH}. As $\Delta$-points pass to superspaces, it follows that $\ell_\infty^*$ contains $\Delta$-points. Now as $\ell_\infty^*$ is known to be isometrically isomorphic to some $L_1(\mu)$-space, it follows from \cite[Corollary~4.1]{MPRZ} that these point are actually super Daugavet points.

\end{proof}

Observe that Corollary~\ref{cor:complemented_renormings}~\ref{item:dual_c0} applies to any infinite dimensional Lipschitz-free space.  Moreover, a positive answer to the following question would immediately provide some improvements to this result.

\begin{question}

Can the space $\ell_1$ be renormed to have a super Daugavet point?

\end{question}

Let us recall that it is unknown whether the Daugavet molecule from Veeorg's space \cite{VeeorgStudia} is a super $\Delta$ or a super Daugavet point.



\end{document}